\documentclass{amsart}
\usepackage[utf8]{inputenc}
\usepackage{enumitem}
\usepackage{amsfonts}
\usepackage{amsmath,amsthm}
\usepackage{comment}
\usepackage{listings}
\usepackage{graphics}
\usepackage{bm}
\usepackage[inline]{asymptote}

\usepackage{graphicx,xcolor}
\graphicspath{{./figures/}}
\usepackage{hyperref,url}
\definecolor{darkblue}{rgb}{0,0,0.75}
\definecolor{darkred}{rgb}{0.75,0,0}
\definecolor{darkgreen}{rgb}{0,0.75,0}   
\hypersetup{colorlinks,
  filecolor=black,
  linkcolor=darkblue,
  citecolor=darkgreen,
  urlcolor=darkred,
  bookmarksopen=true}
\usepackage{amsthm}
\usepackage{graphicx}
\usepackage{amssymb}
\usepackage[section]{placeins}
\usepackage{cite}
\usepackage{appendix}
\usepackage{indentfirst}
\usepackage{extarrows}

\newtheorem{thm}{Theorem}[section]
\newtheorem{lem}[thm]{Lemma}
\newtheorem{prop}[thm]{Proposition}
\newtheorem{cor}[thm]{Corollary}
\theoremstyle{definition}
\newtheorem{defi}[thm]{Definition}
\newtheorem{rmk}[thm]{Remark}
\newtheorem{clm}[thm]{Claim}

\swapnumbers

\numberwithin{equation}{section}

\newcommand{\s}{{\Bar{s}}}
\newcommand{\gb}{\Bar{\gamma}}
\newcommand{\tb}{\Bar{T}}
\newcommand{\nb}{\Bar{N}}
\newcommand{\kb}{\Bar{k}}
\newcommand{\tp}{{t^\prime}}
\newcommand{\fapt}{\forall t\in \left(0,T\right)}
\newcommand{\att}{(\cdot,t)}

\newcommand{\gbb}{\Bar{\Bar{\gamma}}}
\newcommand{\sbb}{\Bar{\Bar{s}}}

\title{Curve Shortening Flow of Space Curves with Convex Projections}
\author{Qi Sun}
\date{\today}
\address{Qi Sun, Department of Mathematics, University of Wisconsin-Madison}
\email{qsun79@wisc.edu}
\begin{document}
\begin{abstract}
We show that under Space Curve Shortening flow any closed immersed curve in $\mathbb R^n$ whose projection onto  $\mathbb{R}^2\times\{\vec{0}\}$  is convex remains smooth until it shrinks to a point. Throughout its evolution, the projection of the curve onto $\mathbb{R}^2\times\{\vec{0}\}$ remains convex.

As an application, we show that any closed immersed curve in $\mathbb R^n$ can be perturbed to an immersed curve in $\mathbb R^{n+2}$ whose evolution by Space Curve Shortening shrinks to a point.
\end{abstract}
\maketitle
\section{Introduction}
\indent   
\subsection*{Background}
Let us consider the Curve Shortening flow (CSF) in higher codimension:
\begin{equation}
    \gamma_t=\gamma_{ss}
\end{equation}
where $\gamma:S^1\times \left[0,T\right)\rightarrow\mathbb{R}^n$ is smooth ($S^1=\mathbb R/2\pi\mathbb Z$), $u\rightarrow \gamma(u,t)$ is an immersion and $\partial_s=\frac{\partial}{\partial s}$ is the derivative with respect to arc-length, defined by
$$\frac{\partial}{\partial s}:=\frac{1}{|\gamma_u|}\frac{\partial}{\partial u}.$$
Let $k\geq 0$ be the curvature, $T$ be the unit tangent vector and $N$ be the unit normal vector of $\gamma$.
At an inflection point ($k=0$), $N$ is not well defined for space curves, but $kN=T_s=\gamma_{ss}$ still makes sense.

For the planar case, the well known Gage-Hamilton theorem \cite{GageHamilton} says that convex curves shrink to a point and become asymptotically circular. Gage and Hamilton's proof built on earlier work by Gage \cite{Gage1,Gage2} and was extended by Grayson in \cite{Grayson}. Grayson showed that embedded planar curves become convex. Since then, many other proofs of the Gage-Hamilton-Grayson theorem have been discovered, notably, \cite{Hamilton+1996+201+222,huisken1998distance,Andrewsnoncollapsing,AndrewsBryan+2011+179+187,andrews2011comparison}.

Curve Shortening flow in higher codimension has been studied in a number of papers including \cite{AltschulerGrayson,Altschuler,altschuler2013zoo,hättenschweiler2015curve,benes2020longterm}. 
See also \cite{he2012distance,yang2005curve,ma2007curve,khan2012condition,corrales2016non,minarvcik2019comparing,pan2021singularities,litzinger2023singularities}.
For results of mean curvature flow in high codimension, we refer to the pioneering work \cite{wang2002long,andrews2010mean}, lecture notes \cite{wang2011lectures} and a survey \cite{smoczyk2011mean}.
      
The higher dimensional case is much less understood than the planar case. As pointed out by Altschuler in \cite[Figure 1]{Altschuler}, preservation of embeddedness, a prerequisite of Grayson's theorem, does not hold 
in general in higher dimensions. However, \emph{projection convexity}, a more restrictive condition than embeddedness, is preserved by CSF in $\mathbb{R}^n$. 
As far as we know, the earliest result in this direction is given in \cite[$\S3.9$]{hättenschweiler2015curve} and then in \cite{benes2020longterm}.
Our purpose in this paper is to investigate along these lines and our results generalize the first part of the Gage-Hamilton theorem to space curves with one-to-one convex projections onto $\mathbb{R}^2$. \\

\subsection*{Notation} 
\subsubsection*{Projection curve}     
Let $P_{xy}:\mathbb{R}^n=\mathbb{R}^2\times\mathbb R^{n-2}\rightarrow\mathbb{R}^2$ be the orthogonal projection onto the first two coordinates, which we call $x$ and $y$. And let
$P_{xy}|_\gamma:\gamma\rightarrow xy$-plane be its restriction to the space curve $\gamma$.
      
Let us define the \emph{projection curve} in the $xy$-plane:
      $$\gb:=P_{xy}\circ\gamma$$ 
      
      A bar on any variable will denote its association to the projection curve $\gb$. 
      To illustrate, $\Bar{k}$, $ \Bar{T}$, $\Bar{N}$, $\Bar{s}$ denote the curvature,  the unit tangent vector, the unit normal vector and the arc-length parameter of $\gb$. We sometimes abuse notation and also use $\gb$ to denote the image of the projection curve.

\subsubsection*{Planar convexity}
\begin{defi}
\label{definitions of three notions of convexity}
    For a continuous curve $\Bar{\gamma}:S^1\rightarrow\mathbb{R}^2$,\\
    $(a)$ A continuous curve $\Bar{\gamma}$ is said to be \emph{convex} if it is simple and its union with its interior is a convex set. \\
    $(b)$ An immersed curve $\Bar{\gamma}$ is said to be \emph{strictly convex} if it intersects each line in the plane in at most two points.\\
    $(c)$ An immersed curve $\Bar{\gamma}$ is said to be \emph{uniformly convex} if it is simple and curvature $\Bar{k}>0$.
\end{defi}

In part $(a)$, the interior is well defined because of the Jordan curve theorem. In part $(b)$, 
a strictly convex curve is actually simple because otherwise one can find a line that passes through one self-intersection of $\gb$ and another point on $\gb$. Such a line would pass through at least three points of $\gb$, which contradicts the definition.

Uniform convexity implies strict convexity, which implies convexity, but not the other way around. Actually, convex curves can contain straight line segments but strictly convex curves cannot. 
And the curvature of strictly convex curves can vanish at some points, which does not happen for uniformly convex curves.

\subsubsection*{Verticality and horizontality}
\begin{defi}
A tangent line of a curve in $\mathbb R^n$ $(n\geq3)$ is said to be \emph{vertical} if it is perpendicular to the $xy$-plane.
\end{defi}
\begin{defi}
A vector in $\mathbb R^n$ $(n\geq3)$ is said to be \emph{horizontal} if it is parallel to the $xy$-plane.
\end{defi}
\begin{rmk}
\label{remarks about vertial tangent lines}
The projection curve $\Bar{\gamma}$ can fail to be immersed when $\gamma$ is immersed but has a vertical tangent line. 
And this will keep us from analyzing $\gamma$ via its projection $\gb$, since the curvature $\kb$ of $\gb$ might not make sense.\\
\end{rmk}    
    
\subsection*{Statement of the main results}
\begin{thm}
\label{main theorem}
Consider an embedded smooth curve $\gamma_0$ in $\mathbb{R}^n(n\geq 2)$, for which $P_{xy}|_{\gamma_0}$ is injective and the projection curve $\Bar{\gamma}_0$ is convex.
     Let $\gamma:S^1\times \left[0,T\right)\rightarrow\mathbb{R}^n$ be the solution to the space CSF with $\gamma(u,0)=\gamma_0(u)$. Then we have:
  \begin{enumerate}[label=(\alph*)]
     \item \label{Main theorem, preservation of convexity} 
      For each $ t\in \left(0,T\right)$, $P_{xy}|_{\gamma\att}$ is injective, $\gamma(\cdot,t)$ has no vertical tangent lines and the projection curve $\Bar{\gamma}\att$ is uniformly convex. 
     \item As $t\rightarrow T$, $\gamma\att$ shrinks to a point.
  \end{enumerate}
\end{thm}

\subsubsection*{Proof strategy}
Our proof mainly relies on constructing barriers (Lemma \ref{construction of the barrier}, Lemma \ref{not double line segment}).
This barrier was used in \cite[Proof of Lemma 2.8 on Page 83]{grayson1989shortening} to compare with curvature and recently was also used in \cite[Lemma 4.3]{mitake2024quenching} as a subsolution. We also use this barrier to get gradient estimates (Lemma \ref{gradient bounds for x direction}, Lemma \ref{gradient bounds}).

\subsubsection*{Supportive examples}
\cite[Section 3]{altschuler2013zoo} provides explicit examples of space CSF that have one-to-one convex projections by solving ODE. Such concrete evolutions of curves support our Theorem \ref{main theorem}. See also \cite{stryker2019construction} where Stryker and Sun rediscovered those examples. 

\subsubsection*{Features of CSF with convex projections}
For a curve $\gamma\att$ satisfying part \ref{Main theorem, preservation of convexity} of Theorem \ref{main theorem}, 
$\gamma\att$ is an embedded curve with $k\att>0$ for all $t\in \left(0,T\right)$. This was already pointed out in \cite{hättenschweiler2015curve}\cite{benes2020longterm}. Without the convex projection assumption, inflection points ($k=0$) can develop even when the curvature 
of the initial curve is positive at all points, as suggested in   \cite[the last paragraph of page 286]{AltschulerGrayson} by an example.

     For curves without one-to-one convex projections, the property that $\gamma$ has no vertical tangent lines does not hold for all positive time in general. In Lemma \ref{detailed analysis that example that no vertical tangent lines is not preserved} we will exhibit an explicit example of a solution of CSF which implies the following result. 
     \begin{lem}
     \label{example that no vertical tangent lines is not preserved}
    There exists a curve $\gamma_0\subset\mathbb R^3$ that has no vertical tangent lines, while $\gamma(\cdot,t_0)$ has a vertical tangent line at some time $t_0\in(0,T)$.
    Actually $\gamma_0$ can be taken to have a locally convex but not one-to-one projection onto the $xy$-plane.
\end{lem}

\subsection*{Previous results of H{\"a}ttenschweiler and Minar\v{c}\'{\i}k \& Bene\v{s}}

As demonstrated by \cite[Theorem 3.59]{hättenschweiler2015curve},
if $P_{xy}|_\gamma$ is injective and the projection curve $\Bar{\gamma}$ is uniformly convex at the initial time then it remains so and shrinks to a line segment or a point. 
Later \cite{benes2020longterm} showed that when $n=3$, if $P_{xy}|_\gamma$ is injective and the projection curve $\Bar{\gamma}$ is convex at initial time, it remains so under the extra assumption that all tangent lines of $\gamma(\cdot,t)$ are not vertical for all time $t\in\left[0,T\right)$. Particularly, curves in \cite{benes2020longterm} are slightly more general because the projection curves $\gb$ are not necessarily uniformly convex however the authors of \cite{benes2020longterm} require an extra assumption. It seems that authors of \cite{benes2020longterm} were unaware of the earlier work \cite{hättenschweiler2015curve} and they used a different method.

To prove that projection convexity is preserved, \cite{benes2020longterm} assume 
\emph{in advance} that $\gamma(\cdot,t)$ has no vertical tangent lines \emph{for all $t\in \left[0,T\right)$}. We show that it suffices to assume that the initial curve has a one-to-one convex projection, and that no assumptions on the solutions at positive times are needed. In \cite{hättenschweiler2015curve} H{\"a}ttenschweiler had done the same under the assumption that the initial curve has a one-to-one \emph{uniformly convex} projection.

To compare with \cite[Theorem 3.59]{hättenschweiler2015curve}, most importantly, we rule out the possibility that the limit curve is a line segment. 

In addition, we only assume the initial projection curve $\gb(\cdot,0)$ to be convex while \cite[Theorem 3.59]{hättenschweiler2015curve} assumes $\gb(\cdot,0)$ to be uniformly convex. Also \cite[Theorem 3.59]{hättenschweiler2015curve} is assuming implicitly that the initial curve $\gamma(\cdot,0)$ has no vertical tangent lines.
Our result holds even when the initial curve $\gamma(\cdot,0)$ has vertical tangent lines, 
which is a previously unconsidered case, as far as we know. In other words, we allow the initial projection curve $\gb(\cdot,0)$ to be non-immersed, such as a polygon, as long as the space curve $\gamma(\cdot,0)$ is immersed.

In summary, our Theorem \ref{main theorem} is new in two aspects. We consider more general curves and give a more precise description on limit curves.

\subsection*{Perturbing planar figure eights}
Our original motivation for considering curves with convex projections comes from the question of how small initial perturbations 
$(\cos u,\epsilon \sin u, \sin2u)$ (the right figure in Figure \ref{A planar figure eight and its perturbation})
of the planar figure eight $(\cos u,0, \sin2u)$ (the left figure in Figure \ref{A planar figure eight and its perturbation}) evolve. The red curve is the corresponding projection curve $(\cos u,\epsilon \sin u, 0)$ in the $xy$-plane.
\begin{figure}[h]
    \centering
\includegraphics[width=1\linewidth]{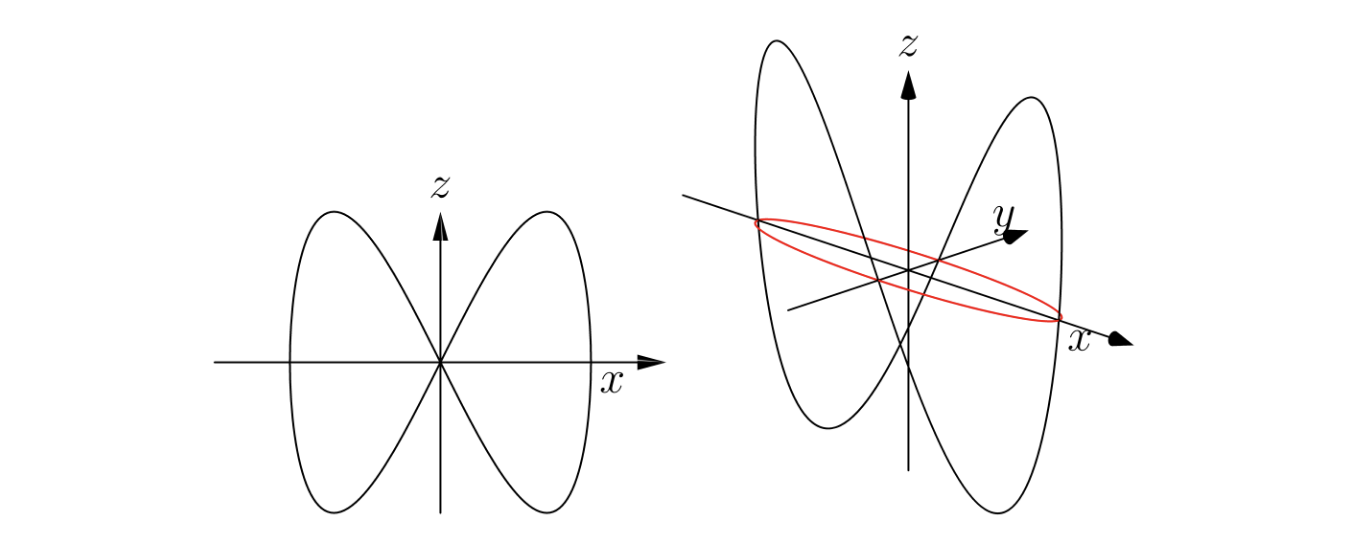}
\caption{A planar figure eight and its perturbation}
\label{A planar figure eight and its perturbation}
\end{figure}

\begin{figure}[ht!]
\centering
\begin{minipage}{0.3\textwidth}
\centering
\includegraphics[width=\linewidth]{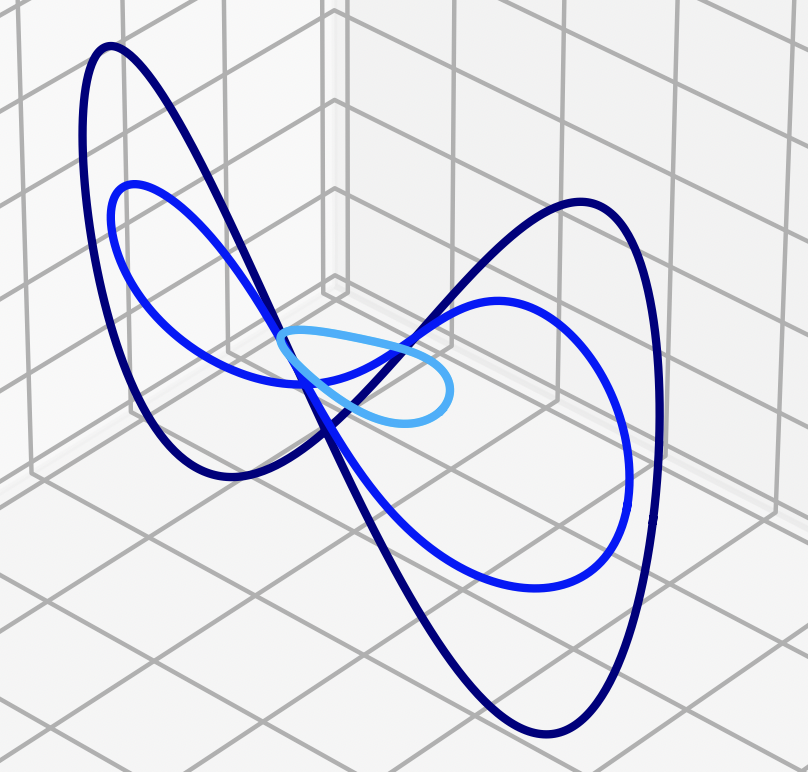}
\end{minipage}\hfill
\begin{minipage}{0.3\textwidth}
\centering
\includegraphics[width=\linewidth]{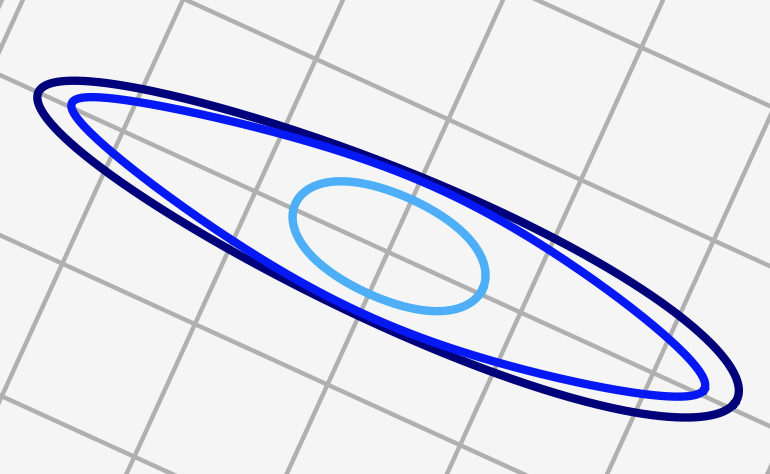}
\textbf{$xy$-projection}
\end{minipage}\hfill
\begin{minipage}{0.3\textwidth}
\centering
\includegraphics[width=\linewidth]{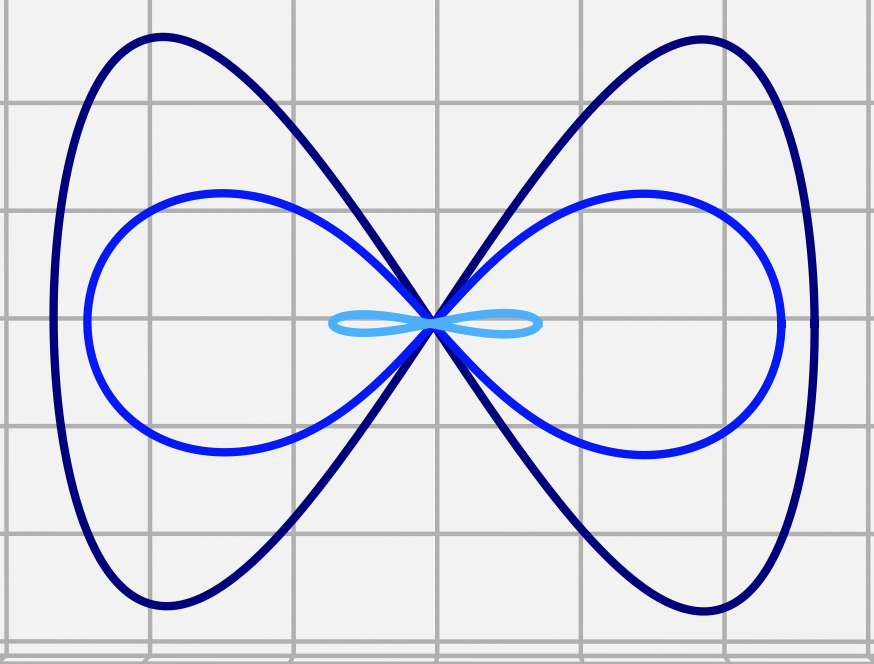}
\textbf{$xz$-projection}
\end{minipage}
\caption{Snapshots of the evolution of a perturbation of the planar figure eight curve from different angles}
\label{fig:three_images}
\end{figure}
It seems interesting to ask whether curves satisfying the conditions of 
Theorem \ref{main theorem} would become asymptotically circular. 
Numerically, such curves including $$(\cos u,\epsilon \sin u, \sin2u)$$ do become asymptotically circular (See Figure \ref{fig:three_images}). 
If a proof can be found that aligns with our numerical observations, such perturbations give examples of how type II singularities 
in Space Curve Shortening can be removed by perturbing the initial conditions.

While we do not know how to analyze all the curves satisfying the conditions of 
Theorem \ref{main theorem}, in a forthcoming paper, we will show for some subcases that curves do develop type I singularities and become asymptotically circular.

\subsection*{Applications}
By embedding $\mathbb R^n$ in $\mathbb R^{n+2}$, we are able to perturb immersed curves to have convex projections, and so that the perturbed curves will shrink to points according to Theorem \ref{main theorem}.
\begin{lem}
\label{perturb general immersed curves with additional codimension two}
For any closed immersed curve $\gamma_0:S^1\rightarrow\mathbb R^n$ with parameter $u$ ($|\gamma_{0u}|\neq0$), the perturbation $\gamma^{\epsilon}_0:S^1\rightarrow\mathbb R^2\times\mathbb R^n$
\begin{equation}
\label{wave approximation}
    \gamma^{\epsilon}_0:=(\epsilon\cos u, \epsilon\sin u,\gamma_0(u))
\end{equation}
has a one-to-one convex projection onto the $xy$-plane and hence shrinks to a point under space CSF for $\epsilon\neq0$.
\end{lem}
Equation (\ref{wave approximation}), referred to as the wave approximation, was introduced in \cite[Definition 5.28]{hättenschweiler2015curve}.
\begin{figure}[ht]
    \centering
\includegraphics[width=0.8\linewidth]{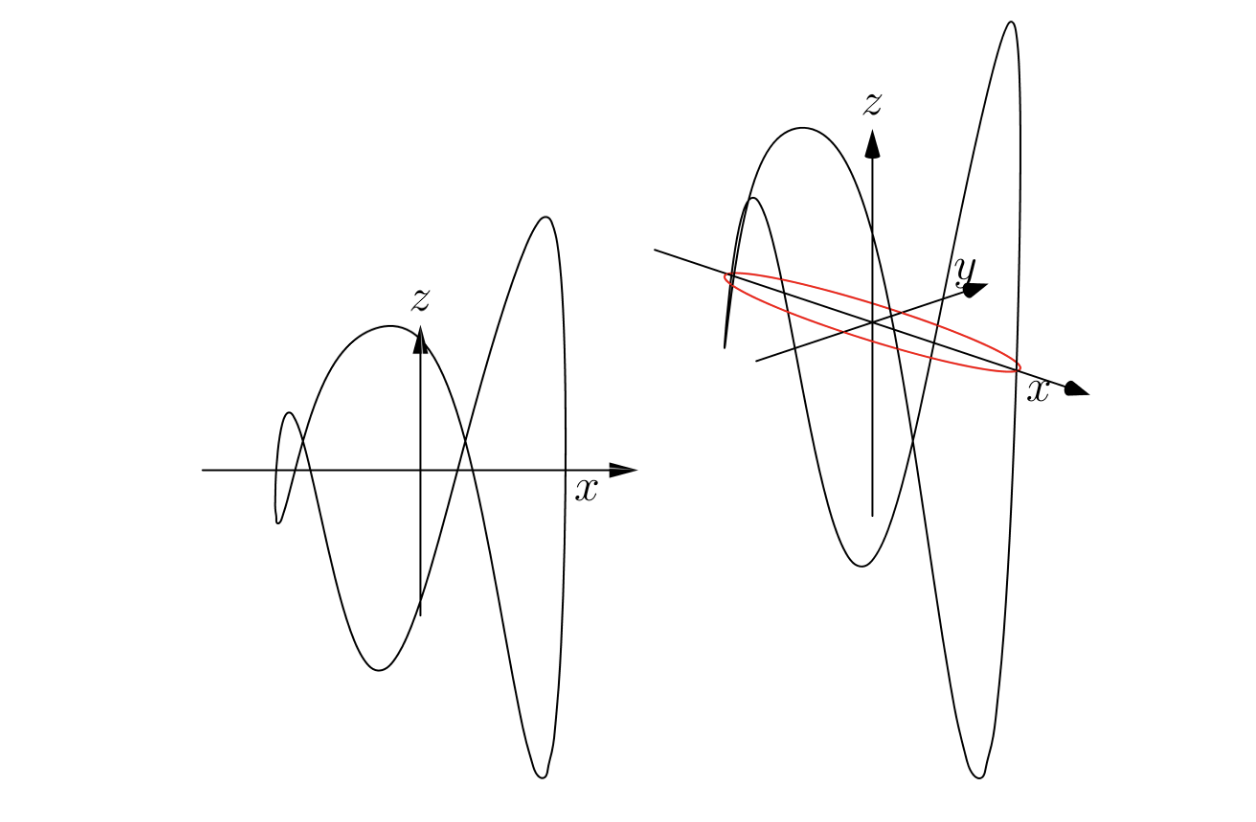}
\caption{An illustration of Lemma \ref{perturb curves with one direction exactly two critical points}.
On the left we have the unperturbed curve in $\mathbb R^2$, and on the right is the perturbation in $\mathbb{R}^3$. Its projection, drawn in red, is an ellipse, hence convex.  By Theorem \ref{main theorem},  the perturbed curve remains smooth until it shrinks to a point.}
\label{Fig:curve with one direction exactly two critical points}
\end{figure}

In the following special case, it suffices to perturb the given curve in $\mathbb R^{n+1}$ instead of $\mathbb{R}^{n+2}$.

\begin{lem}
\label{perturb curves with one direction exactly two critical points}
For any closed immersed curve $\gamma_0:S^1\rightarrow\mathbb R^n$ that admits a nonzero vector $\bm{v}\in\mathbb{R}^n$, such that $\bm{v}\cdot\gamma_0$ has exactly two critical points, there exists a perturbation $\gamma^{\epsilon}_0:S^1\rightarrow\mathbb R^1\times\mathbb R^n$ such that $\gamma^{\epsilon}_0$ has a one-to-one convex projection and hence shrinks to a point for $\epsilon\neq0$.
\end{lem}
\begin{proof}[Proof of Lemma \ref{perturb curves with one direction exactly two critical points}]
We may label the coordinates of $\mathbb R^1\times\mathbb R^n$ to be $(x,y,z_1,\cdots,z_{n-1})$ and may assume $(0,1,0,\cdots,0)\cdot\gamma_0$ has exactly two critical points.  So $P_{xy}(\gamma_0)$ is a double line segment. So we may perturb the space curve $\gamma_0$ such that $P_{xy}(\gamma_0^\epsilon)$ is a convex curve.
\end{proof}

Particularly, perturbations described above also work for some planar figure eights with unequal area, which become singular without shrinking to points under CSF by \cite[Lemma 3]{grayson1989shape}.

\subsection*{Outline of the paper}
$\S2$ and $\S3$  contain the proof of Theorem \ref{main theorem}$(a)$ for all $n\geq2$, using the Sturmian theorem \cite{angenent1988zero} and the strict maximum principle, where Lemma \ref{no tangent lines perpendicular to the horizontal plane} is pivotal.
    
   $\S4$ restricts to the $n=3$ case, introduces the concept of the three-point condition and contains related lemmas.
    The three-point condition quantitatively measures how far space curves with convex projections deviate from planar convex curves. 

    $\S5$ contains the proof of Theorem \ref{main theorem}$(b)$ that $\gamma$ shrinks to a point, assuming that the projection curve $\gb$ shrinks to a point (Proposition \ref{long time behaviour under three point condition}). Our argument is based on the fact that the three-point condition is preserved. We first deal with the $n=3$ case and then the $n>3$ case by  projecting the curves onto $\mathbb R^3$ and studying the corresponding flow (Lemma \ref{evolution of the projection curve in the 3 dim space}),
    using techniques identical to those in the $n=3$ case. 

    $\S6$ summarizes some facts about the Space Curve Shortening flow.
    
    $\S7$ presents the proof of Proposition \ref{long time behaviour under three point condition}. In other words, we show that the projection curve $\gb$ shrinks to a point. This is achieved through constructing barriers and a priori estimates of parabolic equations. 
    
\subsection*{Acknowledgement:} The author is deeply grateful to Sigurd Angenent for his invaluable guidance and constant support. The author would also like to thank Hung Vinh Tran for several helpful discussions and for his encouragement.

\section{Convexity of the projection}
In this section we use the Sturmian theorem\cite{angenent1988zero} to show that $\gb\att$ is convex if the initial projection curve $\gb(\cdot,0)$ is convex.

\subsection*{Curve Shortening Calculus}
Let $\gamma:\mathbb R\times[0,T)\rightarrow\mathbb R^n$ be a family of space curves. We always assume $\gamma$ is continuous, and that the partial derivatives $\gamma_t,\gamma_u,\gamma_{uu}$ exist and are continuous for $t>0$. Moreover we assume that each curve $\gamma\att$ $(t>0)$ is an immersion, that is to say,
$$\gamma_u(u,t)\neq0\text{ for all }u\in\mathbb R, t>0.$$
Throughout the paper we assume that $\gamma\att$  is a closed curve, in other words,
$$\gamma(u+2\pi,t)=\gamma(u,t)\text{ for all }u, t.$$
We define the arc-length derivative to be 
\begin{equation}
\label{definition of partial s}
\frac{\partial}{\partial s}:=\frac{1}{|\gamma_u|}\frac{\partial}{\partial u}.
\end{equation}
The unit tangent vector and the curvature vector are
$$T=\frac{\partial \gamma}{\partial s} \text{ and } \Vec{k}=\frac{\partial^2 \gamma}{\partial s^2}.$$
The normal velocity vector of the curve is the component of $\gamma_t$ that is perpendicular to $T$:
$$\gamma_t^\perp=\gamma_t-(\gamma_t\cdot T)T$$

Let us consider CSF:
$$\gamma_t=\gamma_{ss}$$
\begin{lem}[Lemma 1.4 of \cite{AltschulerGrayson}]
\label{commute of operators}
The operators $\frac{\partial}{\partial t}$ and $\frac{\partial}{\partial s}$ commute according to the following rule:
\begin{equation}
    \frac{\partial}{\partial t}\frac{\partial}{\partial s}= \frac{\partial}{\partial s}\frac{\partial}{\partial t}+k^2\frac{\partial}{\partial s}
\end{equation}    
\end{lem}

\subsection*{Sign-changing number}
Following \cite{angenent1988zero}, we adopt the following definition of the 
sign-changing number:
\begin{defi}[Sign-changing number]
    For a function $h: S^1\rightarrow \mathbb R$, let the sign-changing number of the function $h$ be the supremum over all $k$ such that there are consecutive points $u_1,\cdots,u_k\in S^1$, satisfying
            $$h(u_i)h(u_{i+1})<0, i=1,\cdots, k-1$$ 
    and 
            $$h(u_k)h(u_1)<0.$$
\end{defi}

One can derive from  the Sturmian theorem\cite{angenent1988zero} (see also \cite{angenent1991nodal}, \cite[Proposition 1.2]{Angenent1991surfaces}) that:
\begin{lem}
\label{simple zeros}
    For a fixed nonzero vector $\bm{v}\in \mathbb{R}^n$,
    
    $(a)$ The sign-changing number of $\bm{v} \cdot \gamma_s$ is a non-increasing function of time $t$. 
     
    $(b)$ The number of zeros of $\bm{v} \cdot \gamma_s$ is finite at all positive times $t$.
    
    $(c)$ There exists $t_0\geq0$, such that $\forall t\in\left(t_0,T\right)$,
    all zeros of $\bm{v} \cdot \gamma_s(\cdot,t)$ are simple. Moreover, if $\bm{v} \cdot \gamma_s(\cdot,0)$ changes its sign only twice, 
    then $t_0$ can be chosen to be $0$ and $\forall t\in\left(0,T\right)$, $\bm{v} \cdot \gamma_s(\cdot,t)$ has exactly two zeros, which are simple.
\end{lem}
\begin{proof}
    Let us consider the height function $h:=\bm{v} \cdot \gamma$, by Lemma \ref{commute of operators},
    we have $h_{st}=h_{ts}+k^2h_s=h_{sss}+k^2h_s$.
    Statements $(a)$ $(b)$ follows from \cite{angenent1988zero}.
    
    Considering $(a)$ $(b)$, there exists $t_0$ such that $\forall t\in\left(t_0,T\right)$,
    the number of zeros of $\bm{v} \cdot \gamma_s(\cdot,t)$ is constant.
    So the zeros are simple by \cite{angenent1988zero}.
    
    When $\bm{v} \cdot \gamma_s(\cdot,0)$ changes sign twice, $\bm{v} \cdot \gamma_s(\cdot,t)$ changes sign at most twice due to $(a)$.
    Since $\gamma$ is an immersed closed curve, $\bm{v} \cdot \gamma_s(\cdot,t)$ 
    has at least two sign-changing zeros corresponding to the maximum and minimum points of $\bm{v} \cdot \gamma(\cdot,t)$.
    So $(c)$ follows.\\
\end{proof}

\subsection*{A sufficient condition for planar convexity}
Following \cite[Section 3]{milnor1950total}, we use the following notion.
\begin{defi}
    For a continuous closed curve $\gamma\subset\mathbb{R}^n$ and a unit length vector $\bm{v}\in\mathbb R^n$, let $\mu(\gamma,\bm{v})$ be the number of local maximum points of the function $\bm{v}\cdot\gamma: S^1\rightarrow \mathbb R$.
\end{defi}

The following lemma follows from combining \cite[Theorem 3.1]{milnor1950total} and \cite[Theorem 3.4]{milnor1950total}.
\begin{lem}
\label{a lemma from Milnor}
For a continuous planar curve $\gamma\subset\mathbb{R}^2$, if $\mu(\gamma,\bm{v})=1$ for all unit length vectors $\bm{v}\in\mathbb R^2$, then $\gamma$ is a convex curve.\\
\end{lem}

\subsection*{Projection convexity and no vertical tangent lines}
The following lemma is crucial in the proof of Theorem \ref{main theorem}$(a)$, in particular because it implies that $\fapt$, the projection curve $\gb\att$ is immersed, 
which makes it possible to consider the evolution equations of $\gb$.
\begin{lem}
\label{no tangent lines perpendicular to the horizontal plane}
    Under the conditions of Theorem \ref{main theorem}, $\forall t\in \left(0,T\right)$, 
     $P_{xy}|_{\gamma\att}$ is injective, $\Bar{\gamma}(\cdot,t)$ is convex,
     and $\gamma(\cdot,t)$ has no vertical tangent lines. 
\end{lem}

\begin{proof}
    Let $\bm{v}$ be a nonzero horizontal vector, we have $\gb\cdot \bm{v}=\gamma\cdot \bm{v}$.
    
    Considering that the initial projection curve $\Bar{\gamma}(\cdot,0)$ is convex, the function $\bm{v}\cdot\gamma_s(\cdot,0)$ only changes sign twice. By Lemma \ref{simple zeros}$(c)$, $\forall t\in(0,T)$, $\bm{v} \cdot \gamma_s(\cdot,t)$ 
    has exactly two zeros, which are simple. So these two zeros are the global maximum and minimum point
    of $\bm{v} \cdot \gamma(\cdot,t)$.

    For the projection curve $\gb\att\subset\mathbb R^2$, it follows that $\mu(\gb\att,\bm{v})=1$ for $t\in(0,T)$ and all unit length vectors $\bm{v}\in\mathbb R^2$.
    Lemma \ref{a lemma from Milnor} implies that the continuous closed curve $\gb\att\subset\mathbb R^2$ is convex. Therefore $\gb\att$ is a simple curve and $P_{xy}|_\gamma$ is injective.

We now show that for $t\in\left(0,T\right)$, $\gamma(\cdot,t)$ has no vertical tangent lines. Assume this is not true. Then there exists $u$, such that $\bm{v} \cdot \gamma_s(u,t)=0$, for all unit horizontal vectors $\bm{v}$.
        So for each such $\bm{v}$, $u$ is either the global maximum or minimum point of $\bm{v} \cdot \gamma(\cdot,t)$.
        Let us now define:
    $$A_M:=\{\bm{v}\in S^1\subset\mathbb R^2|(\bm{v} \cdot \gamma)(u,t)\geq(\bm{v} \cdot \gamma)(u^\prime,t),\forall u^\prime\}$$
    $$A_m:=\{\bm{v}\in S^1\subset\mathbb R^2|(\bm{v} \cdot \gamma)(u,t)\leq(\bm{v} \cdot \gamma)(u^\prime,t),\forall u^\prime\}$$
where $S^1$ refers to the set of all unit horizontal vectors.

    We have:
    \begin{itemize}
        \item $A_M\cup A_m=S^1$, since $u$ is either max or min point.
        \item $A_M\neq\emptyset, A_m\neq\emptyset$, since $\bm{v}\in A_M\Leftrightarrow -\bm{v}\in A_m$.
        \item $A_M, A_m$ are closed sets in $S^1$, by definitions.
        \item $A_M\cap A_m=\emptyset$, for if $\bm{v}\in A_M\cap A_m$, then $\bm{v} \cdot \gamma$ would be a constant,
        and thus $\bm{v} \cdot \gamma_s(\cdot,t)\equiv0$. But $\bm{v} \cdot \gamma_s(\cdot,t)$ should only have two zeros.  
    \end{itemize}
    Those four items contradict the connectedness of $S^1$, so this claim is proved.
\end{proof}

\subsection*{A solution that develops a vertical tangent}
We now present the explicit example mentioned in Lemma \ref{example that no vertical tangent lines is not preserved}.

Consider a family of curves 
$$
\gamma^{\epsilon}_0(u):=((\cos u+1-\epsilon) \cos u,(\cos u+1-\epsilon) \sin u, \sin u),
$$
where $u\in[0,2\pi], \epsilon\in[0,1)$. 
\begin{figure}[hb]
\centering
\includegraphics[width=1\linewidth]{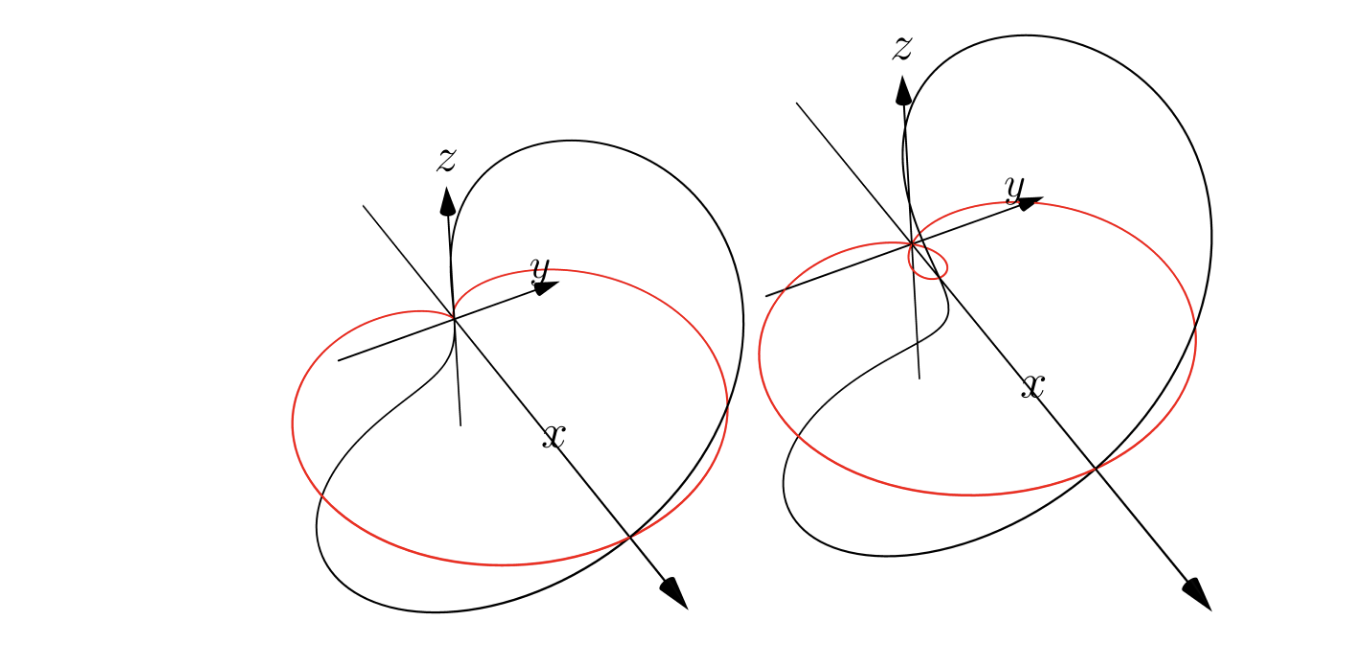}
\caption{The left figure shows the curve $\gamma^{\epsilon}_0$ when $\epsilon=0$ and the right figure depicts the curve when $\epsilon$ is small and positive. The red curves are the corresponding projection curves in the $xy$-plane.}
\end{figure}

\begin{lem}
One can check directly that:

$(a)$ The space curve $\gamma_0^\epsilon$ is smoothly immersed for all $\epsilon$.

$(b)$ The space curve $\gamma_0^\epsilon$ has a vertical tangent when $\epsilon=0$.

$(c)$ The projection curve $P_{x y}\left(\gamma_{0}^{\epsilon}\right)$ is the standard cardioid when $\epsilon=0$.
\end{lem}
\begin{proof}
By direct computation, 
\begin{equation}
\label{derivatives of the explicit example about emerge of the vertical tangent line}   
\gamma_{0u}^{\epsilon}(u)=\left((\cos u+1-\epsilon)(-\sin u)-\sin u \cos u,(\cos u+1-\epsilon) \cos u-\sin ^2 u, \cos u\right).
\end{equation}
And
$$|\gamma_{0u}^{\epsilon}|^2
=(\cos u+1-\epsilon)^2+1\geq 1.$$
Thus part $(a)$ holds.

For part $(b)$, the tangent vector
$\gamma_{0u}^{0}(\pi)=(0,0,-1)$ is vertical.
\end{proof}

Let $\gamma^{\epsilon}:S^1\times \left[0,T^{\epsilon}\right)\rightarrow\mathbb{R}^3$ be the solution of space CSF with $\gamma^{\epsilon}(u,0)=\gamma^{\epsilon}_0(u)$.
\begin{lem}
\label{detailed analysis that example that no vertical tangent lines is not preserved}
There exists a constant $\epsilon_0>0$ such that for $\epsilon\in(0,\epsilon_0)$,

$(a)$ The space curve $\gamma_0^\epsilon$ has no vertical tangent lines.

$(b)$ There exists a time $t_\epsilon\in(0,T^\epsilon)$ such that the space curve $\gamma^\epsilon(\cdot,t_\epsilon)$ does have a vertical tangent line.
    
As a result, we can take $\gamma_0$ in Lemma \ref{example that no vertical tangent lines is not preserved} to be $\gamma_0^{\epsilon}$, for $\epsilon\in(0,\epsilon_0)$.
\end{lem}
\begin{proof}
By equation (\ref{derivatives of the explicit example about emerge of the vertical tangent line}),
$$
\left|P_{x y}\left(\gamma_{0u}^{\epsilon}\right)\right|^2=(\cos u+1-\epsilon)^2+\sin^2u.
$$

So for $\epsilon>0$, $\left|P_{x y}\left(\gamma_{0u}^{\epsilon}\right)\right|^2>0$ for all $u$ and for $\epsilon=0$, $\left|P_{x y}\left(\gamma_{0u}^{\epsilon}\right)\right|^2=0$ iff $u=\pi$.

When $\epsilon$ is small, $\gamma_0^{\epsilon}$ can be viewed as a small $C^1$ perturbation of $\gamma^0_0$ since
$|\gamma^{\epsilon}_0(u)-\gamma^0_0(u)|=\epsilon$ and $|\gamma^{\epsilon}_{0u}(u)-\gamma^0_{0u}(u)|=\epsilon$.

Let us first analyze $\gamma^0$, the evolution of curve $\gamma^0_0$. 

The plane $y=0$ intersects the curve $\gamma^0_0$ at exactly two points, which remain on the $x$-axis under CSF because $\gamma^0$ is $\mathbb Z_2$ rotationally symmetric with respect to the $x$-axis, which implies that the curvature vectors at these two points are parallel to the $x$-axis.

By the Sturmian theorem \cite{angenent1988zero}, except for these two points on the $x$-axis, no other points on $\gamma^0\att$ can touch the planes $y=0$. Thus the plane $y=0$ intersects $\gamma^0\att$ at exactly two points for all $t\in[0,T)$. The same holds for the plane $y=2z$. The plane $y=0$ and $y=2z$ divide $\mathbb R^3$ into four open regions. A point on $\gamma_0^0$ in one region remains in the same region since it cannot touch the plane $y=0$ or $y=2z$. 

By rotational symmetry about the $x$-axis, for $t>0$, $\gamma_{u}^0(\pi,t)$ is parallel to the $yz$-plane.

By \cite{angenent1988zero}, for 
$t>0$, $\gamma_{u}^0(\pi,t)$ is not parallel to the plane $y=0$. Actually   $\gamma_{u}^0(\pi,t)\cdot(0,1,0)<0$ and $\gamma_{u}^0(\pi,t)\cdot(0,0,1)<0$ because a point on $\gamma_0^0$ but not on the $x$-axis remains in the same region.

It holds that $\gamma_{0u}^{\epsilon}(\pi)=\left(0,\epsilon, -1\right)$ and when $\epsilon$ small, $\gamma_0^{\epsilon}$ can be viewed as a small $C^1$ perturbation of $\gamma^0_0$. Fix an $\epsilon$ small, there exists a time $t_\epsilon$ such that 
$\gamma_{u}^{\epsilon}(\pi,t_\epsilon)$ is parallel to $z$-axis because $\gamma_{u}^{\epsilon}(\pi,t)$ remains parallel to the $yz$-plane (because of the symmetry) and sweeps from having a positive $y$ coordinate (because $\gamma_{0u}^{\epsilon}(\pi)=\left(0,\epsilon, -1\right)$) to a negative $y$ coordinate because of the smooth dependence on the initial condition as well as $\gamma_{u}^0(\pi,t)\cdot(0,1,0)<0$. This is the vertical tangent line that we are looking for.

By direct computation, curvature of the projection curve $P_{xy}(\gamma^{\epsilon}_0)$
has no zeros for $\epsilon>0$, which means the projection curve $P_{xy}(\gamma^{\epsilon}_0)$ is locally convex.\\
\end{proof}

\section{Uniform convexity of the projection}
Next, we proceed to prove that the projection curve $\Bar{\gamma}(\cdot,t)$ is uniformly convex for all $t\in \left(0,T\right)$ using the strict maximum principle.

By definition of the arc-length parameter $\Bar{s}$ of $\gb$,
\begin{equation}
\label{definition of partial s bar}
    \frac{\partial}{\partial \s}:=\frac{1}{|\gb_u|}\frac{\partial}{\partial u}.
\end{equation}

The Frenet-Serret formulas in this case are:
\begin{equation}
\label{Frenet-Serret formulas of the projection curve}
    \tb_\s=\kb\nb, \hspace{1cm} \nb_\s=-\kb\tb
\end{equation}

\begin{defi}
Let us call 
\begin{equation}
c=c(s,t):=x_s^2+y_s^2    
\end{equation} the coefficient of the evolution of the projection curve $\gb$.
\end{defi}

The arc-length parameter $s$ of the space curve $\gamma$ parametrizes the projection curve $\gb=(x,y)$. Thus $|\gb_s|=\sqrt{x_s^2+y_s^2}$ and
\begin{equation}
\label{relation between partial s and partial s bar}
\partial_s
=\sqrt{x_s^2+y_s^2}\partial_{\bar{s}}
=\sqrt{c} \partial_{\bar{s}}.
\end{equation}

By direct computations:
\begin{lem}
\label{projection of the normal vector}
For a space curve $\gamma\subset\mathbb{R}^n$ with no vertical tangent lines, one has that  
\begin{equation}
\label{x_ss in terms of sbar}
    P_{xy}(\gamma_{ss})=c\gb_{\bar{s}\bar{s}}+\frac{1}{2}c_{\bar{s}}\gb_{\bar{s}}.
\end{equation}
Thus
$P_{xy}(\gamma_{ss})\cdot\Bar{\gamma}_{\Bar{s}\Bar{s}}=(x_s^2+y_s^2)\Bar{k}^2$.
\end{lem}
\begin{proof}
One has that
\begin{align}
P_{xy}(\gamma_{ss})
&=P_{xy}\left[\sqrt{x_s^2+y_s^2} \left(\sqrt{x_s^2+y_s^2} \gamma_{\bar{s}}\right)_{\bar{s}}\right]\\
&=\sqrt{x_s^2+y_s^2} \left(\sqrt{x_s^2+y_s^2} \gb_{\bar{s}}\right)_{\bar{s}}
=c\gb_{\bar{s}\bar{s}}+\frac{1}{2}c_{\bar{s}}\gb_{\bar{s}}.
\end{align}    
\end{proof}

\subsection*{Evolution of the projection curves}
\begin{lem}[Evolution equation of $\bar{\gamma}$]
\label{evolution equation of the projection curve in the plane}
For $t>0$,
    \begin{equation}
    \label{evolution equation of the projection curve}
       \Bar{\gamma}_t^\perp=(x_s^2+y_s^2)\Bar{k}\Bar{N}=c\Bar{k}\Bar{N}.
    \end{equation}  
\end{lem}
\begin{proof}
By Lemma \ref{projection of the normal vector},
\begin{align}
\gb_t=P_{xy}(\gamma)_t=P_{xy}(\gamma_t)=P_{xy}(\gamma_{ss})
=c\gb_{\bar{s}\bar{s}}+\frac{1}{2}c_{\bar{s}}\gb_{\bar{s}}.
\end{align}
\end{proof}

Up to tangential motion, we may consider the flow
    \begin{equation}
    \label{evolution of the projection curve up to tangential motion}
       \Bar{\gamma}_{t^{\prime}}=(x_s^2+y_s^2)\Bar{k}\Bar{N}=c\Bar{k}\Bar{N},
    \end{equation}
where $\frac{\partial}{\partial t^\prime}$ is chosen such that $\partial_{t^\prime}\gb\perp\partial_{\bar{s}}\gb$. To be more concrete,
\begin{equation}
\label{definition of partial t prime}
    \frac{\partial}{\partial t^\prime}=\frac{\partial}{\partial t}-\frac{c_s}{2c}\frac{\partial}{\partial s}.
\end{equation}

Before proceeding, we explain the notation
$\frac{\partial}{\partial t},\frac{\partial}{\partial t^\prime},\frac{\partial}{\partial s},\frac{\partial}{\partial \bar{s}}$. All functions we consider are defined on some subset of the $(u,t)$-plane. We think of $\frac{\partial}{\partial t},\frac{\partial}{\partial t^\prime},\frac{\partial}{\partial s},\frac{\partial}{\partial \bar{s}}$ as vectorfields on the $(u,t)$-plane, where $\frac{\partial}{\partial t}$ is defined naturally and $\frac{\partial}{\partial t^\prime},\frac{\partial}{\partial s},\frac{\partial}{\partial \bar{s}}$ are defined by equations (\ref{definition of partial t prime}), (\ref{definition of partial s}), (\ref{definition of partial s bar}).

\begin{lem}
We can reparameterize $\gb\att$:
$$\gb(v,t)=\gb(u(v,t),t)$$
such that for $t\in(0,T)$, $|\gb_v|\neq0$ and
\begin{equation}
    \frac{\partial}{\partial t^\prime}\frac{\partial}{\partial v}
    =\frac{\partial}{\partial v}\frac{\partial}{\partial t^\prime}.
\end{equation}
Moreover, for $t\in(0,T)$, it holds that,
\begin{equation}
\label{time derivative of square of gamma bar u}
\left(|\gb_v|^2\right)_{t^\prime}=-2c\Bar{k}^2|\gb_v|^2.
\end{equation} 
\end{lem}
\begin{proof}
For an $\epsilon>0$ fixed, the curve $\gb(\cdot,\epsilon)$ is an immersed curve, say, with parameter $v$. Immersion implies that $|\gb_v(v,\epsilon)|\neq0$.

Let us track each point $p$ on $\gb_0$ under the evolution of equation (\ref{evolution of the projection curve up to tangential motion}) and denote by $p_t\in\gb\att$ the point where  the point $p$ evolves to be at time $t$. For $\gb(\cdot,t), t\in(0,T)$, rechoose the parametrization by labeling the point $p_t$ the same parameter as $p_\epsilon$. In other words, we use the parametrization of the curve $\gb(\cdot,\epsilon)$ to parameterize $\gb(\cdot,t), t\in(0,T)$. Thus the value of the parameter is preserved under $\frac{\partial}{\partial t^\prime}$, which implies that
\begin{equation}
    \frac{\partial}{\partial t^\prime}\frac{\partial}{\partial v}
    =\frac{\partial}{\partial v}\frac{\partial}{\partial t^\prime}.
\end{equation}

Let us show that $\gb$ remains immersed under  this parametrization. Assume otherwise $\gb$ is only immersed on $[\epsilon,t_0)$, for some $t_0<T$ and $\epsilon$ is the fixed constant we chose at the start of this proof. 

When $|\gb_v|\neq0$, by direct computation,
\begin{equation}
\left(|\gb_v|^2\right)_{t^\prime}=2\gb_v\cdot\gb_{vt^\prime}
    =2\gb_v\cdot\gb_{t^\prime v}=2\gb_v\cdot\left(c\Bar{k}\Bar{N}\right)_v
    =2\gb_v\cdot \left(c\Bar{k}\Bar{N}_v\right)
\end{equation}
\begin{equation*}
    =2|\gb_v|\gb_v\cdot \left(c\Bar{k}\Bar{N}_\s\right)
    =2c\Bar{k}|\gb_v|\gb_v\cdot \left(-\kb\tb\right)
    =-2c\Bar{k}^2|\gb_v|^2. 
\end{equation*} 
In summary, 
\begin{equation}
\label{projection curves remain immersed with this parametrization}
\left(|\gb_v|^2\right)_{t^\prime}=-2c\Bar{k}^2|\gb_v|^2.
\end{equation}
Thus on $[\epsilon,t_0)$,
\begin{equation}
\left(\ln\left(|\gb_v|^2\right)\right)_{t^\prime}=-2c\Bar{k}^2.
\end{equation}
So 
\begin{equation}
|\gb_v|^2(v,t_0)=|\gb_v|^2(v,\epsilon)e^{-2\int_\epsilon^{t_0}c\Bar{k}^2(v,t^\prime)dt^\prime}. 
\end{equation}
The integral in the right hand side of the equation is finite because we know the projection curve is smooth on $(0,T)$ which contains $[\epsilon,t_0]$. Therefore $|\gb_v|^2(v,t_0)>0$. 

As a result, $\gb\att$ is immersed for $t\in[\epsilon,T)$ with the parameter $v$.
Since the right hand side of equation (\ref{projection curves remain immersed with this parametrization}) is negative, $|\gb_v|^2(v,t)>0$ is true for all $t\in(0,T)$.
\end{proof}
Doing computations as in  \cite[Section 3]{GageHamilton}, let us now derive various equations associated to this flow (equation (\ref{evolution of the projection curve up to tangential motion})).

We can compute how operators $\frac{\partial}{\partial \tp}$ and $\frac{\partial}{\partial \s}$ commute,
\begin{lem}
\label{commute sbar and tp}
For $t>0$,
\begin{equation}
    \frac{\partial}{\partial \tp}\frac{\partial}{\partial \s}= \frac{\partial}{\partial \s}\frac{\partial}{\partial \tp}+c\kb^2\frac{\partial}{\partial \s}.
\end{equation}
\end{lem}
\begin{proof}
    For a function $f$,
    \begin{equation*}
        f_{\s\tp}=\left(\frac{f_v}{|\gb_v|}\right)_\tp
        =\frac{f_{v\tp}}{|\gb_v|}+f_v\left(\frac{1}{|\gb_v|}\right)_\tp
        =\frac{f_{\tp v}}{|\gb_v|}+f_v(-\frac{1}{2})\frac{\left(|\gb_v|^2\right)_\tp}{|\gb_v|^3}.
    \end{equation*}
By plugging equation (\ref{time derivative of square of gamma bar u}) into it,
\begin{equation*}
f_{\s\tp}=\frac{f_{\tp v}}{|\gb_v|}+f_v(-\frac{1}{2})\frac{-2c\Bar{k}^2|\gb_v|^2 }{|\gb_v|^3}
=\frac{f_{\tp v}}{|\gb_v|}+f_v\frac{c\Bar{k}^2}{|\gb_v|}
=f_{\tp\s}+c\kb^2f_\s.
\end{equation*}
\end{proof}

The time derivative of $\tb$ is given by:
\begin{lem}
For $t>0$,
\label{derivative of tbar}
    \begin{equation}
         \tb_\tp=(c\Bar{k})_\s\nb.
    \end{equation}
\end{lem}
\begin{proof}
By definition of $\tb$ and Lemma \ref{commute sbar and tp},
\begin{equation*}
\tb_\tp=\gb_{\s\tp}=\gb_{\tp\s}+c\kb^2\tb.
\end{equation*}
Considering equation (\ref{evolution of the projection curve up to tangential motion}),
    \begin{equation*}
        \tb_\tp=(c\Bar{k}\Bar{N})_\s+c\kb^2\tb
        =(c\Bar{k})_\s\nb+(c\Bar{k})\Bar{N}_\s+c\kb^2\tb.
    \end{equation*}
It follows from the Frenet-Serret formulas(equation (\ref{Frenet-Serret formulas of the projection curve})),
    \begin{equation*}
        \tb_\tp=(c\Bar{k})_\s\nb+(c\Bar{k})(-\kb\tb)+c\kb^2\tb
        =(c\Bar{k})_\s\nb.
    \end{equation*}
\end{proof}

Let $\theta$ be the angle between the unit tangent vector $\tb$ and the positive $x$-axis, we have:
\begin{lem}
\label{time derivative of theta}
For $t>0$,
    \begin{equation}
        \theta_\tp=(c\Bar{k})_\s.
    \end{equation}
\end{lem}
\begin{proof}
    By definition $\tb=(\cos\theta,\sin\theta)$ and $\nb=(-\sin\theta,\cos\theta)$.
    
    By direct computation,
    \begin{equation*}
        \tb_\tp=(-\sin\theta,\cos\theta)\theta_\tp.
    \end{equation*}
    Lemma \ref{derivative of tbar} implies that
    \begin{equation*}
        \tb_\tp=(c\Bar{k})_\s\nb=(c\Bar{k})_\s(-\sin\theta,\cos\theta).
    \end{equation*}
    This lemma is proved by comparing these two equations.
\end{proof}

We now can derive:
\begin{lem}[Evolution equation of $\Bar{k}$]
\label{changing tangential motion, evolution of projection curve in the plane}
For $t>0$,
    \begin{equation}
    \label{evolution equation of k bar}
       \Bar{k}_{t^{\prime}}={(c\Bar{k})}_{\Bar{s}\Bar{s}}+c\Bar{k}^3=
       c\Bar{k}_{\Bar{s}\Bar{s}}+2c_{\Bar{s}}\Bar{k}_{\Bar{s}}+(c_{\Bar{s}\Bar{s}}+c\Bar{k}^2)\Bar{k}.
    \end{equation} 
\end{lem}

\begin{proof}
    Since $\gb\att$ is a planar curve, $\kb=\theta_\s$.
    
    By Lemma \ref{commute sbar and tp} and Lemma \ref{time derivative of theta},
    \begin{equation*}
        \kb_\tp=\theta_{\s\tp}=\theta_{\tp\s}+c\kb^2\theta_\s
        =(c\Bar{k})_{\s\s}+c\kb^3.
    \end{equation*}
\end{proof}

To apply the Sturmian theorem\cite{angenent1988zero} and the maximum principle to the 
    equation (\ref{evolution equation of k bar}), it is indispensable that the second-order coefficient $c$ is strictly positive. Actually,
\begin{lem}
\label{explanation that the coefficient c is positive}
    For all $t\in(0,T)$, it holds that $c(\cdot,t)>0$.
\end{lem}
\begin{proof}
In Lemma \ref{no tangent lines perpendicular to the horizontal plane} we proved that the tangent lines to $\gamma\att$ are never vertical when $t>0$. Therefore $c(u,t)>0$ for all $t\in(0,T)$.
\end{proof}

\subsection*{Proof of Theorem \ref{main theorem}(a)}

It follows from Lemma \ref{no tangent lines perpendicular to the horizontal plane} and Lemma \ref{explanation that the coefficient c is positive} that $\forall t\in \left(0,T\right)$, $c(\cdot,t)>0$, $\Bar{\gamma}(\cdot,t)$ is immersed and convex. 
Notice $\Bar{k}(\cdot,t)\geq0$ because of convexity. The remaining part is to prove that $\Bar{\gamma}(\cdot,t)$ is uniformly convex.
    
    For each $\epsilon>0$, 
    $-M:=\inf\limits_{t\in\left[\epsilon,T-\epsilon\right]}(c_{\Bar{s}\Bar{s}}+c\Bar{k}^2)>-\infty$. 
    So equation (\ref{evolution equation of k bar}) implies that on $\left[\epsilon,T-\epsilon\right]$,
\begin{align}
\Bar{k}_{t^{\prime}}&\geq c\Bar{k}_{\Bar{s}\Bar{s}}+2c_{\Bar{s}}\Bar{k}_{\Bar{s}}-M\Bar{k}\\
&=c\frac{1}{|\gb_v|}\left(\frac{\kb_v}{|\gb_v|}\right)_v+2c_{\Bar{s}}\frac{1}{|\gb_v|}\Bar{k}_v-M\Bar{k}\\
&=c\frac{\kb_{vv}}{|\gb_v|^2}+\left[c\frac{1}{|\gb_v|}\left(\frac{1}{|\gb_v|}\right)_v+2c_{\Bar{s}}\frac{1}{|\gb_v|}\right]\Bar{k}_v-M\Bar{k}\\
&=A(v,t)\kb_{vv}+B(v,t)\kb_v-M\Bar{k}.
\end{align}
    
    We compare $\bar{k}$ with the solution to
    $$h_{t^\prime}=-Mh,h(\epsilon)=0.$$
    Notice $h\equiv0$ is the solution to this ODE. 
    
Since $\inf\limits_{t\in\left[\epsilon,T-\epsilon\right]}c(\cdot,t)>0$, 
on time interval $[\epsilon,T-\epsilon]$, the coefficient $A$ is bounded from below by some positive constant. By the strong maximum principle (see for example \cite{protter2012maximum}), $\Bar{k}(\cdot,t)>0$, $\forall t\in\left(\epsilon,T-\epsilon\right]$, $\forall\epsilon>0$.
    Thus $\forall t\in\left(0,T\right)$, $\Bar{\gamma}(\cdot,t)$ is uniformly convex.

\section{The three-point condition}
In this section, let us restrict to curves in $\mathbb{R}^3$ and introduce the three-point condition, which will play a central role in our argument in section 4. According to \cite[Chapter IV]{Rado1933Plateau}, this three-point condition first appeared in a short announcement of Hilbert and then was used by Lebesgue to develop a first attack on the problem of Plateau in the non-parametric form.

We will discuss a sufficient condition for the three-point condition. As far as we know, this sufficient condition was first proposed in \cite{hartman1966bounded}. We will give a direct proof of this sufficient condition in our setting.\\

\subsubsection*{Slope of lines and planes in $\mathbb R^3$}
\begin{defi}[Slope of a line]
    For a line $L\subset\mathbb{R}^3$, the slope $S_L$ of the line $L$ is defined to be $\tan\Theta$, 
    where $\Theta$ is the nonnegative acute angle between the line $L$ and the $xy$-plane.
    For a vertical line $L$, $S_L:=+\infty$, and for a horizontal line $L$, $S_L:=0$.
\end{defi}

\begin{defi}[Slope of a plane]
    For a plane $P\subset\mathbb{R}^3$, the slope $S_P$ of the plane $P$ is defined to be $\tan\Theta$, 
    where $\Theta$ is the nonnegative acute dihedral angle between the plane $P$ and the $xy$-plane.
    For a vertical plane $P$, $S_P:=+\infty$, and for a horizontal plane $P$, $S_P:=0$.
\end{defi}

With the above notion of slope, we notice:
\begin{lem}
\label{a comparison of slopes}
    For a line $L$ contained in a plane $P$, we have $S_L\leq S_P$.\\
\end{lem}

\subsubsection*{Intersection number}
\begin{defi}
     For a plane $P$ and a space curve $\gamma$, let us define the intersection number $|P\bigcap\gamma|$ to be $ \#\{u\in S^1|\gamma(u)\in P\}$, where $\#$ counts the number of elements of one set.\\
\end{defi}

\subsubsection*{Secant lines and osculating planes} 
\begin{defi}[Secant lines]
    For a curve $\gamma$ in $\mathbb{R}^3$, we would call a line $l\subset\mathbb{R}^3$ a secant line of the curve $\gamma$ if the line $l$ intersects the curve $\gamma$ at a minimum of two distinct points.
\end{defi}

\begin{defi}[Osculating planes]
    For a curve $\gamma$ in $\mathbb{R}^3$ with positive curvature, for each
    $u\in S^1$, the osculating plane at the point $\gamma(u)$ is defined to be the plane $P(u)$  that passes through the point $\gamma(u)$ and is spanned by the unit tangent vector $T(u)$ and the unit normal vector $N(u)$.          
\end{defi}

\begin{rmk}
    We do need the curvature of the curve to be positive to have well-defined normal vectors to define the osculating planes.
\end{rmk}

\subsection*{The three-point condition}
\begin{defi}[The three-point condition\footnote{In comparison with \cite[Lemma 3.58 ii)]{hättenschweiler2015curve} which discusses a similar topic, the novelty of the three-point condition lies in demonstrating that the minimum of the constant $\Delta$ at each time $t$ is monotonic under CSF with respect to the time $t$ while $d_1,d_2$ in \cite[Lemma 3.58)]{hättenschweiler2015curve} are not necessarily monotonic with respect to the time $t$. See Proposition \ref{properties of space curves during evolution} for more details.}]
A  space curve $\gamma$ in $\mathbb{R}^3$ satisfies a three-point condition with a constant $\Delta\in\left[0,+\infty\right)$ (with respect to the $xy$-plane) 
if $S_P\leq\Delta$ for all planes $P\subset\mathbb{R}^3$ satisfying $|P\bigcap\gamma|\geq3$. 
\end{defi}

Let us point out the following sufficient condition for the three-point condition, as a corollary of \cite{hartman1966bounded}. We will give a direct proof in our setting at the end of this section.
\begin{lem}
\label{sufficient conditions for three-point condition}
Let $\gamma\subset\mathbb{R}^3$ be an embedded smooth curve that has no vertical tangent lines. If $P_{xy}|_\gamma$ is injective and the projection curve $\Bar{\gamma}$ is uniformly convex,
    then $\gamma$ satisfies a three-point condition with some constant $\Delta\in\left[0,+\infty\right)$.
\end{lem}

\begin{rmk}
    Actually, \cite{hartman1966bounded} used a different definition of uniformly convex planar curves
    (see \cite[(1.3) on Page 512]{hartman1968convex}). It can be proved to be equivalent to our definition \ref{definitions of three notions of convexity}$(c)$ for smooth curves.\\
\end{rmk}

\subsection*{Relations of several closely related conditions}
For a curve $\gamma$ in $\mathbb{R}^3$ and a constant $\Delta\in\left[0,+\infty\right)$,
we label the following conditions\footnote{We will show that $(\Gamma_1)\Leftrightarrow(\Gamma_1^\prime)$ and $(\Gamma_4)\Leftrightarrow(\Gamma_4^\prime)$.}:\\
    $(\Gamma_1)$ Each vertical plane $P\subset\mathbb{R}^3$ intersects $\gamma$ in at most two points.   \\
    $(\Gamma_1^\prime)$ The map $P_{xy}|_\gamma$ is injective and the projection curve $\Bar{\gamma}$ is strictly convex.     \\
    $(\Gamma_2)$   The curve $\gamma$ satisfies a three-point condition with some constant $\Delta$.  \\
    $(\Gamma_3)$    Slopes of all secant lines of $\gamma$ are bounded by some constant $\Delta$.\\
    $(\Gamma_4)$    Slopes of all tangent lines of $\gamma$ are bounded by some constant $\Delta$.\\
    $(\Gamma_4^\prime)$ The curve $\gamma$ has no vertical tangent lines.\\
    $(\Gamma_5)$   The curvature $k$ is positive and the slopes of all osculating planes are bounded by some constant $\Delta$.\\

   Let us discuss their relations. 
    
    First, by comparing the definitions directly,
\begin{lem}    $(\Gamma_2)\Rightarrow(\Gamma_1)$.
\end{lem}
\begin{proof}
The definition of the three-point condition is equivalent to the requirement that the intersection number $|P\bigcap\gamma|\leq2$ for any plane $ P$ with slope $S_P>\Delta$.
\end{proof}

Furthermore, 
\begin{prop}
\label{comparison of different properties of space curves}
The following conclusion holds:\\
$(a)$ $(\Gamma_1)\Leftrightarrow(\Gamma_1^\prime)$ and $(\Gamma_4)\Leftrightarrow(\Gamma_4^\prime)$.\\
    $(b)$  A chain of implications: $(\Gamma_2)\Rightarrow(\Gamma_3)\Rightarrow(\Gamma_4)$ with same constant $\Delta$.\\
    $(c)$   If we assume $(\Gamma_1)$, then $(\Gamma_4)\Rightarrow(\Gamma_3)$ with possibly different $\Delta$. \\
    $(d)$   If we assume $(\Gamma_1)$ and $(\Gamma_5)$, then $(\Gamma_3)\Rightarrow(\Gamma_2)$ with possibly different $\Delta$. \\
\end{prop}

Particularly, it follows from Proposition \ref{comparison of different properties of space curves}$(b)$ that,
\begin{cor}
\label{three-point condition implies bounds on the slopes}
For a curve $\gamma\subset\mathbb R^3$ satisfying a three-point condition with a constant $\Delta\in\left[0,+\infty\right)$ and two points $q_1,q_2$ on $\gamma$, say $q_1=(x_1,y_1,z_1),q_2=(x_2,y_2,z_2)$, one has $(z_1-z_2)^2\leq\Delta^2\left[(x_1-x_2)^2+(y_1-y_2)^2\right]$.\\
\end{cor}

\subsection*{Proof of Proposition \ref{comparison of different properties of space curves}}
\begin{proof}[Proof of Proposition \ref{comparison of different properties of space curves}$(a)$]
\leavevmode \par

$(\Gamma_1)\Rightarrow(\Gamma_1^\prime)$:
    
    If $P_{xy}|_\gamma$ is not injective, then there are two different points $q_1,q_2$ on $\gamma$, 
    such that $P_{xy}(q_1)=P_{xy}(q_2)$, pick a random point $q_3\neq q_1,q_2$ on $\gamma$, 
    and let $P$ be the plane passing through the points $q_1,q_2,q_3$. $P_{xy}(q_1)=P_{xy}(q_2)$ implies that $P$ is vertical. Contradiction.

    If $\Bar{\gamma}$ is not strictly convex. There is a line $L$ in the $xy$-plane, such that $|L\bigcap\Bar{\gamma}| \geq 3$.
    Notice $P:=P_{xy}^{-1}(L)$ is a vertical plane. Considering each $u\in S^1$, such that $\Bar{\gamma}(u)\in L$, we have
    $\gamma(u)\in P_{xy}^{-1}(\Bar{\gamma}(u))\subset P_{xy}^{-1}(L)=P$ and we already proved $P_{xy}|_\gamma$ is injective, thus $|P\bigcap\gamma|\geq|L\bigcap\Bar{\gamma}| \geq 3$. 
    Contradiction.

$(\Gamma_1)\Leftarrow(\Gamma_1^\prime)$:
If this were not true, then there would exist a vertical plane $P\subset\mathbb{R}^3$ with $|P\bigcap\gamma|\geq 3$. 
The projection $L:=P_{xy}(P)$ is a line in the $xy$-plane. For each $u\in S^1$ with $\gamma(u)\in P$, we have $\Bar{\gamma}(u)= P_{xy}(\gamma(u))\in P_{xy}(P)=L$. Since $P_{xy}|_\gamma$ is injective, it follows that $|L\bigcap\Bar{\gamma}|\geq|P\bigcap\gamma|\geq3$. Contradiction.

$(\Gamma_4)\Rightarrow(\Gamma_4^\prime)$: 

This is true since a line with bounded slope cannot be vertical.

$(\Gamma_4)\Leftarrow(\Gamma_4^\prime)$:

The condition $(\Gamma_4^\prime)$ implies that $x_s^2+y_s^2>0$. Since $S^1$ is compact, there exists a constant $\Delta\in[0,+\infty)$ such that $\frac{z_s^2}{x_s^2+y_s^2}\leq\Delta^2$. 
\end{proof}

\begin{proof}[Proof of Proposition \ref{comparison of different properties of space curves}$(b)$]\leavevmode \par

$(\Gamma_2)\Rightarrow(\Gamma_3)$:\\ 
If this were not true, then there would exist two different points $q_1,q_2$ on $\gamma$, with $S_{L_{q_1 q_2}}>\Delta$, where $L_{q_1 q_2}$ is the straight line passing through $q_1,q_2$.
Pick a different point $q_3\neq q_1,q_2$ on $\gamma$, and let $P$ be the plane passing through $q_1,q_2,q_3$.
By Lemma \ref{a comparison of slopes}, $S_P\geq S_{L_{q_1 q_2}}>\Delta$. Contradiction.
    
$(\Gamma_3)\Rightarrow(\Gamma_4)$: Tangent lines are limits of secant lines.    
\end{proof}

\begin{proof}[Proof of Proposition \ref{comparison of different properties of space curves}$(c)$]\leavevmode \par
Arguing by contradiction, we assume that there exist two sequences $q_{1n},q_{2n}$ on $\gamma$, such that $S_{L_{q_{1n} q_{2n}}}\rightarrow+\infty$
as $n\rightarrow+\infty$,where $L_{q_{1n} q_{2n}}$ is the straight line passing through $q_{1n},q_{2n}$. Taking subsequences without changing notation, we may assume there exist two points $q_{1\infty},q_{2\infty}$ on $\gamma$, such that 
$q_{1n}\rightarrow q_{1\infty},q_{2n}\rightarrow q_{2\infty}$. 
     
     If $q_{1\infty}\neq q_{2\infty}$, 
     then $S_{L_{q_{1\infty} q_{2\infty}}}=+\infty$, which contradicts $(\Gamma_1^\prime)$. 
     
     If $q_{1\infty}=q_{2\infty}$,
     then the slope of the tangent line at point $q_{1\infty}=q_{2\infty}$ is infinite, which contradicts $(\Gamma_4)$
\end{proof}

\begin{proof}[Proof of Proposition \ref{comparison of different properties of space curves}$(d)$]\leavevmode \par
If Proposition \ref{comparison of different properties of space curves}$(d)$ were false, then we would have three sequences $q_{1n},q_{2n},q_{3n}$ on $\gamma$ such that the slope $S_{P_n}\rightarrow+\infty$, where
$P_n$ is the plane passing through $q_{1n},q_{2n},q_{3n}$. 
Taking subsequences, we may assume there exist three points $q_{1\infty}$, $q_{2\infty}$, $q_{3\infty}$ on $\gamma$, such that $q_{1n}\rightarrow q_{1\infty},q_{2n}\rightarrow
q_{2\infty},q_{3n}\rightarrow q_{3\infty}$.
     
If $q_{1\infty},q_{2\infty},q_{3\infty}$ are three different points, then let $P$ be the plane passing through the points $q_{1\infty},q_{2\infty},q_{3\infty}$. The plane $P$ is vertical since $S_{P_n}\rightarrow+\infty$. This contradicts $(\Gamma_1)$ because the intersection number $|P\cap\gamma|\geq3$.
     
If $q_{1\infty}=q_{2\infty}=q_{3\infty}$, we get a contradiction with $(\Gamma_5)$ since 
     osculating planes can also be defined as the limits of planes passing through three points on the curve, where limits are taken
     as three points converge to one point. 
     
If $q_{1\infty}=q_{2\infty}\neq q_{3\infty}$, let $P$ be the plane containing the tangent line at $q_{1\infty}=q_{2\infty}$ and the point $q_{3\infty}$. 
Because of $(\Gamma_3)$ and the fact that tangent lines are limits of secant lines, the tangent line at $q_{1\infty}=q_{2\infty}$ has a bounded slope. Considering $(\Gamma_1^\prime)$, the plane $P$ is not vertical since it passes through another point $q_{3\infty}$ on the curve. This contradicts $S_{P_n}\rightarrow+\infty$.\\
\end{proof}

\subsection*{Proof of Lemma \ref{sufficient conditions for three-point condition}}
    The curve $\gamma\subset\mathbb{R}^3$ satisfies $(\Gamma_1^\prime)$, $(\Gamma_4^\prime)$ and the curvature of the projection curve is positive ($\kb>0$).

    By $(a)$ of Proposition \ref{comparison of different properties of space curves}, $\gamma$ satisfies $(\Gamma_1)$$(\Gamma_4)$.
        
    \begin{clm}
    \label{positive curvature}
        The curvature of the space curve $\gamma$ is positive($k>0$).
    \end{clm}
\begin{proof}[Proof of the claim]
The condition $(\Gamma_4)$ and $\kb>0$ imply that the right hand side of $P_{xy}(\gamma_{ss})\cdot\Bar{\gamma}_{\Bar{s}\Bar{s}}=(x_s^2+y_s^2)\Bar{k}^2$ (from Lemma \ref{projection of the normal vector}) is positive,
so that $k^2=|\gamma_{ss}|^2$ cannot be zero.
\end{proof}

    \begin{clm}
    \label{claim about gamma 5}
        The space curve $\gamma$ satisfies $(\Gamma_5)$.
    \end{clm}
\begin{proof}[Proof of the claim]
We already know $k>0$, the first part of $(\Gamma_5)$, so that the binormal vector $B(u)$ is well defined.
The osculating plane at $\gamma(u)$ is vertical iff $B(u)\cdot(0,0,1)=0$.
    
    By definition, $$B(u)\cdot(0,0,1)=(T(u)\times N(u))\cdot(0,0,1)=\frac{1}{k}\left|
    \begin{array}{ll}
          x_s & y_s \\
          x_{s s} & y_{s s}
     \end{array}\right|.$$
     
     Thus $B(u)\cdot(0,0,1)=0$ if and only if $ P_{xy}(\gamma_{ss})$  is parallel to $\Bar{T}$, the unit tangent vector of the projection curve.
     
Since $\Bar{T}\perp\Bar{N}$, the vector $\Bar{T} $ is parallel to $ P_{xy}(\gamma_{ss})$ iff $P_{xy}(\gamma_{ss})\perp\Bar{N}$ 
iff $P_{xy}(\gamma_{ss})\cdot\Bar{\gamma}_{\Bar{s}\Bar{s}}=0$
iff (by Lemma \ref{projection of the normal vector}) $x_s^2+y_s^2=0$.

In conclusion, the osculating plane at $\gamma(u)$ is vertical if and only if $x_s^2+y_s^2=0$, which doesn't hold due to $(\Gamma_4)$.
     Thus $\gamma$ has no vertical osculating planes.
     Considering that $S^1$ is compact, we have a finite upper bound for slope of osculating plane.
\end{proof}
    
    By $(c)$ of Proposition \ref{comparison of different properties of space curves}, $\gamma$ satisfies $(\Gamma_3)$.

    By $(d)$ of Proposition \ref{comparison of different properties of space curves} and our Claim \ref{claim about gamma 5}, $\gamma$ satisfies $(\Gamma_2)$.

\bigskip

\section{\texorpdfstring{Proof of Theorem \ref{main theorem}(b)\\ assuming the projection shrinks to a point}{Proof of Theorem \ref{main theorem}(b) assuming the projection shrinks to a point}}

In this section, let us prove our main theorem assuming the following:
\begin{prop}
\label{long time behaviour under three point condition}
       Under the conditions of Theorem \ref{main theorem}, the projection curve $\Bar{\gamma}$ shrinks to a point.
\end{prop}

We will prove Proposition \ref{long time behaviour under three point condition} in $\S7$.
    When $n=2$, our proof of Proposition \ref{long time behaviour under three point condition}
    gives a new proof of the first part of the Gage-Hamilton theorem, to the best of our knowledge.

We first consider the $n=3$ case and then deal with the higher-dimensional case by considering its projections onto subspaces of dimension three.

We will show the preservation of the three-point condition, which would help to control $\gamma$ via its projection $\gb$.
\subsection*{ \texorpdfstring{Case $n=3$}{Case n=3}}
It follows from the Sturmian theorem that,
\begin{lem}
\label{intersection number with a fixed plane}
For any fixed plane $P\subset\mathbb{R}^3$, the intersection number $|P\bigcap\gamma(\cdot,t)|:=\#\{u\in S^1|\gamma(u,t)\in P\}$ is a non-increasing function of time $t$.
\end{lem}
\begin{proof}
The equation of the plane $P$ is of the form $$ax+by+cz=d,$$ where $a,b,c,d\in \mathbb{R}$. 

Here $\gamma(u,t)=(x(u,t),y(u,t),z(u,t))$ satisfies the Curve Shortening flow $\gamma_t=\gamma_{ss}$.

Let us consider the function $f:=ax+by+cz-d$. We have $(f\circ\gamma)_t=(f\circ\gamma)_{ss}$ and notice that the zeros of the function $f\circ\gamma$ correspond to the intersection points.
This lemma follows from applying the Sturmian theorem to the equation $(f\circ\gamma)_t=(f\circ\gamma)_{ss}$.
\end{proof}

The preservation of the condition
$(\Gamma_1)\Leftrightarrow(\Gamma_1^\prime)$ was first discovered by H{\"a}ttenschweiler \cite[Lemma 3.56]{hättenschweiler2015curve} 
who used Lemma \ref{intersection number with a fixed plane} to prove this. To extend this argument further:
\begin{prop}
\label{properties of space curves during evolution}
    If $\gamma(\cdot,0)$ satisfies a three-point condition with some constant $\Delta\in\left[0,+\infty\right)$,
    then $\forall t\in\left[0,T\right)$, $\gamma\att$  satisfies the three-point condition
    with the same constant $\Delta$.
\end{prop}

\begin{proof}[Proof of Proposition \ref{properties of space curves during evolution}]
The definition of the three-point condition is equivalent to requiring $|P\bigcap\gamma|\leq2$ for all $P$ with $S_P>\Delta$. 
So this proposition follows from Lemma \ref{intersection number with a fixed plane}.
\end{proof}

Based on Proposition \ref{long time behaviour under three point condition}, for the case $n=3$ we can now show:
\begin{proof}[Proof of  part $(b)$ of Theorem \ref{main theorem}]
By $(a)$ of Theorem \ref{main theorem} and Lemma \ref{sufficient conditions for three-point condition}, $\fapt$, $\gamma\att$ satisfies a three-point condition. 
In particular, if we fix one $\epsilon\in\left(0,T\right)$, then there exists a constant $\Delta\in\left[0,+\infty\right)$ such that $\gamma(\cdot,\epsilon)$ satisfies the three-point condition with constant $\Delta$.
Proposition \ref{properties of space curves during evolution} then implies that $\gamma\att$ satisfies the three-point condition with the same constant $\Delta$ for all $t\in\left[\epsilon,T\right)$.
    
By Corollary \ref{three-point condition implies bounds on the slopes}, for arbitrary two points $q_1,q_2$ on $\gamma\att$, say $q_1=(x_1,y_1,z_1),q_2=(x_2,y_2,z_2)$, one has
\begin{equation}
\label{height get controlled by the horizontal projection}
(z_1-z_2)^2\leq\Delta^2\left[(x_1-x_2)^2+(y_1-y_2)^2\right]    
\end{equation}
for the same constant $\Delta$ for all $t\in\left[\epsilon,T\right)$.
Thus 
\begin{equation}
\label{diameter control}
\text{diam }\left(\gamma\att\right)\leq\sqrt{1+\Delta^2}\text{diam }\left(\gb\att\right),
\end{equation}
where diam $(\cdot)$ denotes the diameter of the curve.

Proposition \ref{long time behaviour under three point condition} says that diam $\left(\gb\att\right)\rightarrow0$ as $t\rightarrow T$. Therefore equation (\ref{diameter control}) implies that diam $\left(\gamma\att\right)\rightarrow0$ as $t\rightarrow T$.
\end{proof}

\subsection*{Case \texorpdfstring{$n \geq 3$}{n >= 3}}
In this subsection, we modify our argument for higher dimensional space.

For a fixed integer $i\in[1,n-2]$, consider the projection:

$$P_{xyz_{i}}:\mathbb{R}^n\rightarrow\mathbb{R}^{3}, P_{xyz_{i}}(x,y,z_{1},\cdots,z_{n-2})=(x,y,z_i)$$

And the projection curve in the $xyz_i$-space:    
$$\Bar{\Bar{\gamma}}:=P_{xyz_{i}}\circ\gamma$$

We will use $\Bar{\Bar{k}},\Bar{\Bar{T}},\Bar{\Bar{N}},\Bar{\Bar{B}},\Bar{\Bar{s}}$ as associated notation for the projection curve $\Bar{\Bar{\gamma}}$.

Notice that the restricted map $P_{xyz_{i}}|_{\gamma}:$ $\gamma\rightarrow\mathbb{R}^3$ is injective
since  $P_{xy}|_{\gamma}$ is injective.\\

A modified Lemma \ref{projection of the normal vector} holds:
\begin{lem}
\label{projection of the normal vector, n dim case}
    Let $\gamma:S^1\rightarrow\mathbb{R}^n$ be an immersed curve with $x_s^2+y_s^2+z_{is}^2>0$, we have $P_{xyz_i}(\gamma_{ss})\cdot\Bar{\Bar{\gamma}}_{\Bar{\Bar{s}}\Bar{\Bar{s}}}=(x_s^2+y_s^2+z_{is}^2)\Bar{\Bar{k}}^2$ and $P_{xyz_i}(\gamma_{ss})\cdot\Bar{\Bar{B}}=0$.
\end{lem}

\begin{proof}
Noticing that the arc-length parameter $s$ of the space curve $\gamma$ parametrizes the projection curve $\gbb$, by definition of the arc-length parameter $\bar{\Bar{s}}$ of $\gbb$,
    
$$ \partial_s=\sqrt{x_s^2+y_s^2+z_{is}^2} \partial_{\Bar{\Bar{s}}}.$$
So    $$x_{ss}=\sqrt{x_s^2+y_s^2+z_{is}^2} (\sqrt{x_s^2+y_s^2+z_{is}^2} x_{\Bar{\Bar{s}}})_{\Bar{\Bar{s}}}.$$

    Since $(x_{\Bar{\bar{s}}},y_{\Bar{\bar{s}}},z_{i\Bar{\bar{s}}})\cdot(x_{\Bar{\Bar{s}}\Bar{\Bar{s}}},y_{\Bar{\Bar{s}}\Bar{\Bar{s}}},z_{i\Bar{\Bar{s}}\Bar{\Bar{s}}})=0$
    and $\Bar{\Bar{k}}^2=x_{\Bar{\Bar{s}}\Bar{\Bar{s}}}^2+y_{\Bar{\Bar{s}}\Bar{\Bar{s}}}^2+z_{i\Bar{\Bar{s}}\Bar{\Bar{s}}}^2$,
    $$P_{xyz_i}(\gamma_{ss})\cdot\Bar{\Bar{\gamma}}_{\Bar{\Bar{s}}\Bar{\Bar{s}}}=(x_{ss},y_{ss},z_{iss})\cdot(x_{\Bar{\Bar{s}}\Bar{\Bar{s}}},y_{\Bar{\Bar{s}}\Bar{\Bar{s}}},z_{i\Bar{\Bar{s}}\Bar{\Bar{s}}})=0+(x_s^2+y_s^2+z_{is}^2)\Bar{\Bar{k}}^2.$$
    So the first part of this lemma is proved.
    
    For the second part that, notice $P_{xyz_i}(\gamma_{ss})$ can be written as linear combination of $\Bar{\Bar{T}},\Bar{\Bar{N}}$.
\end{proof}

Consider the following coefficient:
$$c_i=c_i(s,t):=x_s^2+y_s^2+z_{is}^2$$

We can also modify Lemma \ref{evolution equation of the projection curve in the plane}. By Lemma \ref{projection of the normal vector, n dim case},
\begin{lem}[Evolution of $\overline{\overline{\gamma}}$]
\label{evolution of the projection curve in the 3 dim space}
    Consider Curve Shortening flow in $\mathbb{R}^n$,  $\gamma:S^1\times \left[0,T\right)\rightarrow\mathbb{R}^n$,
     assume $x_s^2+y_s^2+z_{is}^2>0$, We have:
    \begin{equation}
    \label{evolution equation of the projection curve in the 3 dim space}
       \Bar{\Bar{\gamma}}_t-\Bar{\Bar{\gamma}}_t^{\top}=(x_s^2+y_s^2+z_{is}^2)\Bar{\Bar{k}}\Bar{\Bar{N}}=c_i\Bar{\Bar{k}}\Bar{\Bar{N}},
    \end{equation} 
    where $\Bar{\Bar{\gamma}}_t^{\top}$ is the tangential part of $\Bar{\Bar{\gamma}}_t$.\\
\end{lem}

Up to tangential motion, for the evolution of $\gbb$ (Lemma \ref{evolution of the projection curve in the 3 dim space}), we may consider the following flow in $\mathbb{R}^3$:
\begin{equation}
\label{evolution of projection curves in three dimension space}
\Bar{\Bar{\gamma}}_{t^{\prime\prime}}=(x_s^2+y_s^2+z_{is}^2)\Bar{\Bar{k}}\Bar{\Bar{N}}=c_i\Bar{\Bar{k}}\Bar{\Bar{N}},
\end{equation}
where $\frac{\partial}{\partial t^{\prime\prime}}=\frac{\partial}{\partial t}-\frac{1}{2}\frac{c_{is}}{c_i}\frac{\partial}{\partial s}$.

It follows from the Sturmian theorem that,
\begin{lem}
\label{intersection number of projection curves with a fixed plane}
    Assume $c_i>0$, for any fixed plane $P\subset\mathbb{R}^3$, 
    the intersection number $|P\bigcap\gbb(\cdot,t)|$ is a non-increasing function of time $t$.
\end{lem}
\begin{proof}
The equation of the plane $P$ would be of the form $ax+by+cz_i=d$, where $a,b,c,d\in \mathbb{R}.$\\
Here $\gbb=(x,y,z_i)$ satisfies the evolution equation (\ref{evolution of projection curves in three dimension space}).\\
Let us consider the function $f:=ax+by+cz_i-d$.\\
It holds that $(f\circ\gbb)_{t^{\prime\prime}}=c_i(f\circ\gbb)_{\sbb\sbb}$ and notice that zeros of $f\circ\gbb$ correspond to intersection points.
This lemma then follows from the Sturmian theorem.
\end{proof}

Finally, for the general dimension,
\begin{proof}[Proof of Theorem \ref{main theorem} $(b)$ based on Proposition \ref{long time behaviour under three point condition}]
By Theorem \ref{main theorem}$(a)$, $c_i\att$ $\geq x_s^2+y_s^2=c\att>0$, $\fapt$.
    
    By Lemma \ref{sufficient conditions for three-point condition},
    $\fapt,\gbb\att$ satisfies the three-point condition. Particularly, fix one $\epsilon\in\left(0,T\right)$, there exists a constant $\Delta\in\left[0,+\infty\right)$ such that
    $\gbb(\cdot,\epsilon)$ satisfies the three-point condition with constant $\Delta$.
    Considering Lemma \ref{intersection number of projection curves with a fixed plane},
    $\forall t\in\left[\epsilon,T\right)$, $\gbb\att$ satisfies the three-point condition with the same constant $\Delta$.
    
By Corollary \ref{three-point condition implies bounds on the slopes}, for arbitrary two points $q_1,q_2$ on $\gbb\att$, say $q_1=(x_1,y_1,z_{i1})$, $q_2=(x_2,y_2,z_{i2})$, one has that $$(z_{i1}-z_{i2})^2\leq\Delta^2\left[(x_1-x_2)^2+(y_1-y_2)^2\right]$$
for the same constant $\Delta$ for all $t\in\left[\epsilon,T\right)$.
Right hand side of the inequality goes to zero by Proposition \ref{long time behaviour under three point condition}.

     This argument works for all $i$. So $\gamma$ shrinks to a point.
\end{proof}

\begin{cor}
\label{three-point condition in n dim}
Under conditions of Theorem \ref{main theorem}, $\forall\epsilon>0$, there exists $M>0$ such that for each $t\in\left[\epsilon,T\right)$ and for arbitrary two points $q_1,q_2$ on $\gamma\att$, say $q_1=(x_1,y_1,z_{11},\cdots,z_{(n-2)1})$, $q_2=(x_2,y_2,z_{12},\cdots,z_{(n-2)2})$, one has that
$$\sum_{i=1}^{n-2}(z_{i1}-z_{i2})^2\leq M^2\left[(x_1-x_2)^2+(y_1-y_2)^2\right].$$
\end{cor}
\begin{cor}
\label{lower bounds of c}
    Under conditions of Theorem \ref{main theorem}, $\forall\epsilon>0$, there exists $\delta>0$
    such that $c\att=x_s^2+y_s^2\geq\delta>0$ for $t\in\left[\epsilon,T\right)$.
\end{cor}

\begin{proof}
By Proposition \ref{comparison of different properties of space curves}$(b)$, particularly $(\Gamma_4)$,
$z_{is}^2\leq \Delta_i^2 (x_s^2+y_s^2)$ for some constants $\Delta_i\in[0,+\infty)$ for $t\in\left[\epsilon,T\right)$.

    Notice $1=x_s^2+y_s^2+z_{1s}^2+\cdots+z_{(n-2)s}^2\leq (x_s^2+y_s^2)(1+\Delta_1^2+\cdots\Delta_{(n-2)}^2)$.

One can then choose $\delta=\frac{1}{1+\Delta_1^2+\cdots\Delta_{(n-2)}^2}$.
\end{proof}

\begin{rmk}
This corollary implies that the coefficient $c=x_s^2+y_s^2$ of the evolution equation of the projection curve $\gb$ (equation (\ref{evolution of the projection curve up to tangential motion})),  is bounded from above and below by positive constants.
\end{rmk}

\section{Some facts about the Space Curve Shortening flow}
In this section, we gather the lemmas that will be used.

\subsection*{Long-time existence}
The following lemma follows from \cite[Theorem 1.13]{AltschulerGrayson} and alternatively from estimates
in \cite[Theorem 3.1]{Altschuler}. One may
refer to \cite{yang2005curve} for the $n\geq4$ case. 
\begin{lem}
\label{bounded curvature can extend flow}
If the curvature $k$ is bounded on the time interval $\left[0,t_1\right)$, then there exists an $\epsilon>0$ 
such that $\gamma(\cdot,t)$ exists and is immersed on the extended time interval $\left[0,t_1+\epsilon\right)$.\\
\end{lem}

\subsection*{Finite-time singularities}
   The next lemma follows from the proof of \cite[Lemma 3.1]{altschuler2013zoo}. See also the proof of \cite[Theorem 4.1]{he2012distance}. 
\begin{lem}
Let $B_{R}(0)$ be the ball centered at the origin with radius $R$. If $\gamma(\cdot,0)\subset B_{R}(0)$, then $\gamma(\cdot,t)\subset B_{\sqrt{R^2-2t}}(0)$. 
\end{lem}
\begin{proof}
By direct computation, $$\left(|\gamma|^2\right)_t=2\gamma\gamma_t=2\gamma\gamma_{ss}
=2(\gamma\gamma_{s})_s-2=\left(|\gamma|^2\right)_{ss}-2.$$

So we have $$\left(|\gamma|^2+2t\right)_t=\left(|\gamma|^2+2t\right)_{ss}.$$

The conclusion follows from the maximum principle.
\end{proof}
   \begin{cor}
   \label{singularity time is finite}
       Under the Space Curve Shortening flow, closed curves develop singularities in finite time.\\
   \end{cor}

\subsection*{The limit curve}
For the proof of the next lemma, we refer readers to the proof of \cite[Lemma 3.21]{hättenschweiler2015curve}.
\begin{lem}
\label{uniform convergence}  
As $t\rightarrow T$, $\gamma(\cdot,t)$ converges to a Lipschitz curve denoted by $\gamma(\cdot,T)$,
in the sense of $C^\alpha, 0\leq\alpha<1$. Particularly,
\begin{equation}
    \lim_{t\rightarrow T}\max_{u\in S^1}|\gamma(u,t)-\gamma(u,T)|=0.
\end{equation}
\end{lem}

\subsection*{Curvature bounds}
Based on the established theory of the a priori estimates of parabolic PDEs, curvature bounds can be derived from gradient bounds:
\begin{prop}
\label{curvature bounds}
Let $\mathbf{r}=\mathbf{r}(x,t)=(y(x,t),z_1(x,t),\cdots,z_{(n-2)}(x,t))$  be a solution to the graph flow
\begin{equation}
\label{the equation of the graph flow}
\mathbf{r}_t=\frac{\mathbf{r}_{xx}}{1+\left|\mathbf{r}_x\right|^2} 
\end{equation}
 on $\left[0,M\right]\times\left[0,T\right)$.
Assume that $|\mathbf{r}_x|\leq M_1<+\infty$ on $\left[\delta,M-\delta\right]\times\left[0,T\right)$ for some constant $\delta>0$,
then the curvature $k$ is uniformly bounded on $\left[2\delta,M-2\delta\right]\times\left[0,T\right)$.
\end{prop}
\begin{proof}
By examining each component of the graph flow separately, the function $y$ satisfies the equation:
    $$y_t=\frac{y_{xx}}{1+y_x^2+z_{1x}^2+\cdots+z_{(n-2)x}^2}$$ 
Differentiate both sides with respect to $x$:
    $$(y_x)_t=\left(\frac{1}{1+y_x^2+z_{1x}^2+\cdots+z_{(n-2)x}^2}(y_x)_x\right)_x$$
The function $y_x$ satisfies a parabolic equation in divergence form. 
    
It follows from the Nash-Moser estimates (see, for example \cite{Nashestimates}) 
that $y_x$ is Hölder continuous. Similarly, $z_{ix}$ is Hölder continuous for all $i$ by  examining the other components of the graph flow.

Thus, we can view the equation $$y_t=\frac{y_{xx}}{1+y_x^2+z_{1x}^2+\cdots+z_{(n-2)x}^2}$$
as a linear equation with Hölder continuous coefficient because $y_x, z_{ix}$ and thus the coefficient $\frac{1}{1+y_x^2+z_{1x}^2+\cdots+z_{(n-2)x}^2}$ are Hölder continuous.

The Schauder estimates (see, for example \cite[Theorem 8.11.1]{krylov1996lectures}) yield bounds for second order derivatives of $y$.
Similarly, we can get 
bounds for second order derivatives of $z_i$.

The curvature is
\begin{equation}
\frac{\sqrt{\left(1+|\mathbf{r}_x|^2\right)|\mathbf{r}_{xx}|^2-(\mathbf{r}_x\cdot\mathbf{r}_{xx})^2}}{\left(1+|\mathbf{r}_x|^2\right)^\frac{3}{2}}.    
\end{equation}

Thus the curvature is uniformly bounded since we we have assumed gradient bounds and have derived bounds on the second order derivatives $|\mathbf{r}_{xx}|$.
\end{proof}

\begin{rmk}
    See also \cite[Theorem 4.3]{AltschulerGrayson} which considered the small gradient case with a direct proof without resorting to the theory of parabolic estimates. See also \cite{smoczyk2016curvature}.
\end{rmk}

\section{Proof of Proposition \ref{long time behaviour under three point condition}}
In this section, we use barriers to prove Proposition \ref{long time behaviour under three point condition} by showing that the image of the limit curve $\gb(\cdot,T)$, defined in accordance with Lemma \ref{uniform convergence}, is a single point. The key lemma is Lemma \ref{construction of the barrier}.

Throughout this section, we may assume that the coefficient
$c=x_s^2+y_s^2\geq\delta>0$ and that the projection curve $\gb\att$ is uniformly convex for all $t\in\left[0,T\right)$,
according to Corollary \ref{lower bounds of c} and Theorem \ref{main theorem} $(a)$, 
by selecting $\gamma(\cdot,\epsilon)$ as the initial curve when necessary, for some small $\epsilon>0$. 
\\

Let $D(t)\subset\mathbb{R}^2$ be the closed domain enclosed by the projection curve $\bar{\gamma}\att$
and $\mathring{D}(t)$ be its interior.

Let us first point out that all points on $\bar{\gamma}\att$ are moving inward:
\begin{lem}
\label{$D(t)$ is monotonic}
    For $t_1>t_2$, ${D(t_1)}\subsetneq \mathring{D}(t_2)$.
\end{lem}

\begin{proof}
    This follows from the equation (\ref{evolution of the projection curve up to tangential motion}), 
    $c\geq\delta>0$ and $\Bar{k}>0$ since $\gb$ is uniformly convex.\\
\end{proof}

Let us define:
\begin{equation}
\label{limit domain D}
D:=\bigcap_{t\in\left[0,T\right)}D(t)
\end{equation}

\begin{lem}
\label{properties of D}
The set $D$ is nonempty, compact and convex.
\end{lem}

\begin{proof}
Nonemptiness follows from the Cantor Intersection Theorem.
    
For all $t\in\left[0,T\right)$, $D(t)$ is closed and uniformly bounded. So their intersection $D$ is closed and bounded, thus compact. 
    
    For all $ p,q\in D$, let $\overline{pq}$ be the line segment connecting $p$ and $q$. 
    The points $ p,q\in D$ implies $ p,q\in D(t)$, so $\overline{pq}\subset D(t), \forall t\in\left[0,T\right)$, 
    since $D(t)$ is convex.
    Thus $\overline{pq}\subset D$.
\end{proof}

\begin{lem}
\label{p0q0}
    If $D$ contains more than one point, then there exist two different points $p_0,q_0\in D$, such that $d(p_0,q_0)=\sup\limits_{p,q\in D}d(p,q)$, 
    where $d$ is the standard Euclidean distance on $\mathbb{R}^2$.
\end{lem}

\begin{proof}
        The set $D$ is compact, so $D\times D$ is compact. 
        
        The function $d:D\times D\rightarrow\mathbb{R}$ is a continuous function,  supremum of which can be achieved on a compact set.\\
\end{proof}

\subsection*{\texorpdfstring{Identification of the limit curve and the boundary of $D$}{Identification of the limit curve and the boundary of D}}
For $M$ a metric space, let $F(M)$ be the collection of all non-empty compact subsets of $M$ with the metric being the \emph{Hausdorff distance} $d_H$. It is known that the metric space $F(M)$ is complete when $M$ is complete.

\begin{lem}
\label{converge in the sense of the Hausdorff distance}
    The compact set $D(t)$ converges to $D$, as $t\rightarrow T$, in the sense of the Hausdorff distance $d_H$.
\end{lem}
\begin{proof}

    Consider the map:
    $$D:[0,T)\rightarrow F(D(0))$$
    $$t\rightarrow D(t)$$

    By Lemma \ref{uniform convergence}, the map $D$ is a Cauchy-continuous function, which means that for any Cauchy sequence $(t_1,t_2,\cdots)$ in $[0,T)$, the sequence $(D(t_1),D(t_2),\cdots)$ is a Cauchy sequence in $F(D(0))$. As a result, since $F(D(0))$ is complete, we have a unique extension:
    $$D:[0,T]\rightarrow F(D(0))$$
    $$t\rightarrow D(t)$$
    and $D$ is a uniformly continuous map on $[0,T]$.

    Next, let us show that $D(T)=D$. 

    \begin{clm}
        $D(T)\subset D$.
    \end{clm}
    \begin{proof}[Proof of the claim]
    Since $D(t)$ is shrinking by Lemma \ref{$D(t)$ is monotonic}, we can also consider the map:
    $$D:[t_0,T)\rightarrow F(D(t_0))$$
    $$t\rightarrow D(t)$$
    and its extension to $[t_0,T]$, so $D(T)\subset D(t_0)$, $\forall t_0\in[0,T)$.

    By definition of $D$, $D(T)\subset D$.
    \end{proof}
    \begin{clm}
        $D\subset D(T)$.
    \end{clm}
    \begin{proof}[Proof of the claim]
    Considering $D(t)\rightarrow D(T)$ as $t\rightarrow T$ in $F(D(0))$, $\forall \epsilon>0$, there exists $\delta>0$ such that $\forall t\in[T-\delta,T)$, $D(t)\subset(D(T))_\epsilon$, where $(D(T))_\epsilon$ denotes the $\epsilon$-neighbourhood of $D(T)$.

   Since $D$ is defined to be the intersection of $D(t)$, $D\subset(D(T))_\epsilon$, $\forall \epsilon>0$.

Hence
$D\subset\bigcap\limits_{\epsilon>0}(D(T))_\epsilon= D(T)$.
   \end{proof}

   Lemma \ref{converge in the sense of the Hausdorff distance} is proved since $D(t)\rightarrow D(T)$ as $t\rightarrow T$ in the sense of the Hausdorff distance because $D$ is a continuous function.
\end{proof}

It follows from \cite[Theorem 20]{wills2007hausdorff},
\begin{lem}
\label{Hausdorff distance of convex sets and boundaries}
    For two non-empty compact convex sets $A,B$ in $\mathbb R^2$, $d_H(A,B)=d_H(\partial A, \partial B)$.
\end{lem}

Now we are ready to show:
\begin{lem}
As compact sets, $\gb(S^1,T)=\partial D$.
\end{lem}
\begin{proof}
    By Lemma \ref{Hausdorff distance of convex sets and boundaries} and Lemma \ref{properties of D},
    $$d_H(D(t),D)=d_H(\partial D(t),\partial D).$$
    
    By Lemma \ref{converge in the sense of the Hausdorff distance}, $d_H(D(t),D)\rightarrow0$, thus $d_H(\partial D(t),\partial D)\rightarrow0$.
        
    It follows from the Lemma \ref{uniform convergence} that $\gb\att\rightarrow\gb(\cdot,T)$ uniformly, and hence  $\gb(S^1,t)\rightarrow\gb(S^1,T)$ in Hausdorff distance.

    Now we can take limits for both sides of the equality $\partial D(t)=\gb(S^1,t)$ in the sense of the Hausdorff distance. As a result,  $\partial D=\gb(S^1,T)$.\\
\end{proof}

\subsection*{The upper branch}
\begin{defi}
For a vector $\bm{v}\in\mathbb{R}^n$ and $t\in[0,T)$, 
\begin{equation}
(\bm{v}\cdot\gamma)_{max}(t):=\max_{u\in S^1} \bm{v}\cdot\gamma(u,t)
\end{equation}
and
\begin{equation}
(\bm{v}\cdot\gamma)_{min}(t):=\min_{u\in S^1} \bm{v}\cdot\gamma(u,t).
\end{equation}
\end{defi}

For a vector $\bm{v}\in\mathbb{R}^n$,
by applying the maximum principle to the equation 
\begin{equation}
\label{equation of CSF in direction v}
(\bm{v}\cdot\gamma)_t=(\bm{v}\cdot\gamma)_{ss},    
\end{equation}
the maximum of $\bm{v}\cdot\gamma\att$ is nonincreasing and the minimum of $\bm{v}\cdot\gamma\att$ is nondecreasing. Hence we can define:
\begin{defi}
For a vector $\bm{v}\in\mathbb{R}^n$, 
\begin{equation}
(\bm{v}\cdot\gamma)_{max}(T):=\lim\limits_{t\rightarrow T}(\bm{v}\cdot\gamma)_{max}(t)    
\end{equation}
and
\begin{equation}
(\bm{v}\cdot\gamma)_{min}(T):=\lim\limits_{t\rightarrow T}(\bm{v}\cdot\gamma)_{min}(t)    
\end{equation}  
\end{defi}

Given a family of functions that converge uniformly, the operation of taking the supremum can be interchanged with the limit process. As a result, one has
\begin{align}
(\bm{v}\cdot\gamma)_{max}(T)
&=\lim\limits_{t\rightarrow T}(\bm{v}\cdot\gamma)_{max}(t)
=\lim\limits_{t\rightarrow T}\max_{u\in S^1} \bm{v}\cdot\gamma(u,t)\\
&=\max_{u\in S^1}(\bm{v}\cdot\lim\limits_{t\rightarrow T}\gamma(u,t))
=\max_{u\in S^1}(\bm{v}\cdot\gamma(u,T)),    
\end{align}
where $\gamma(u,T)$ is defined in accordance with Lemma \ref{uniform convergence}.\\

\begin{lem}
\label{properties related to v dot gamma}
For a nonzero horizontal vector $\bm{v}$ and $t\in[0,T)$, 

$(a)$ The function $\bm{v}\cdot\gamma(\cdot,t)$ has exactly two critical points: 
one global maximum point and one global minimum point.  

$(b)$ One has that
\begin{equation}
(\bm{v}\cdot\gamma)_{min}(t)<(\bm{v}\cdot\gamma)_{min}(T)
\leq(\bm{v}\cdot\gamma)_{max}(T)<(\bm{v}\cdot\gamma)_{max}(t)    
\end{equation}
\end{lem}
\begin{proof}
The projection curve $\gb\att$ is uniformly convex and one can apply the strict maximum principle to the equation (\ref{equation of CSF in direction v}).
\end{proof}

Based on Lemma \ref{properties related to v dot gamma} $(a)$, let us denote by $u_{max}(t)$ the value of $u$ where $\bm{v}\cdot\gamma\att$ attains its maximum and by $u_{min}(t)$ where it attains its minimum. We may assume 
\begin{equation}
u_{min}(t)<u_{max}(t)<u_{min}(t)+2\pi,
\end{equation}
for each time $t$.\\

To keep our notation straightforward:
$$x=(1,0,0,\cdots,0)\cdot\gamma\hspace{1cm}y=(0,1,0,\cdots,0)\cdot\gamma$$
both of which take the form $\bm{v}\cdot\gamma$.

By Lemma \ref{properties related to v dot gamma}, for each time $t<T$, the set
\begin{equation}
\label{the set that is the union of the upper and lower branch}
\{\gamma(u,t)|x_{min}(T)\leq x(u,t)\leq x_{max}(T)\}    
\end{equation}
has exactly two connected components because the projection curve $\gb\att$ is embedded. In particular, there exists
\begin{equation}
\label{lablelling of u_i(t)}
u_{min}(t)<u_1(t)<u_2(t)<u_{max}(t)<u_3(t)<u_4(t)<u_{min}(t)+2\pi
\end{equation}
as labelled in Figure \ref{upper branch}, such that 
\begin{align}
x(u_1(t),t)&=x(u_4(t),t)=x_{min}(T)\\
x(u_2(t),t)&=x(u_3(t),t)=x_{max}(T)
\end{align}
\begin{figure}[h]
\centering
\includegraphics[width=0.5\textwidth]{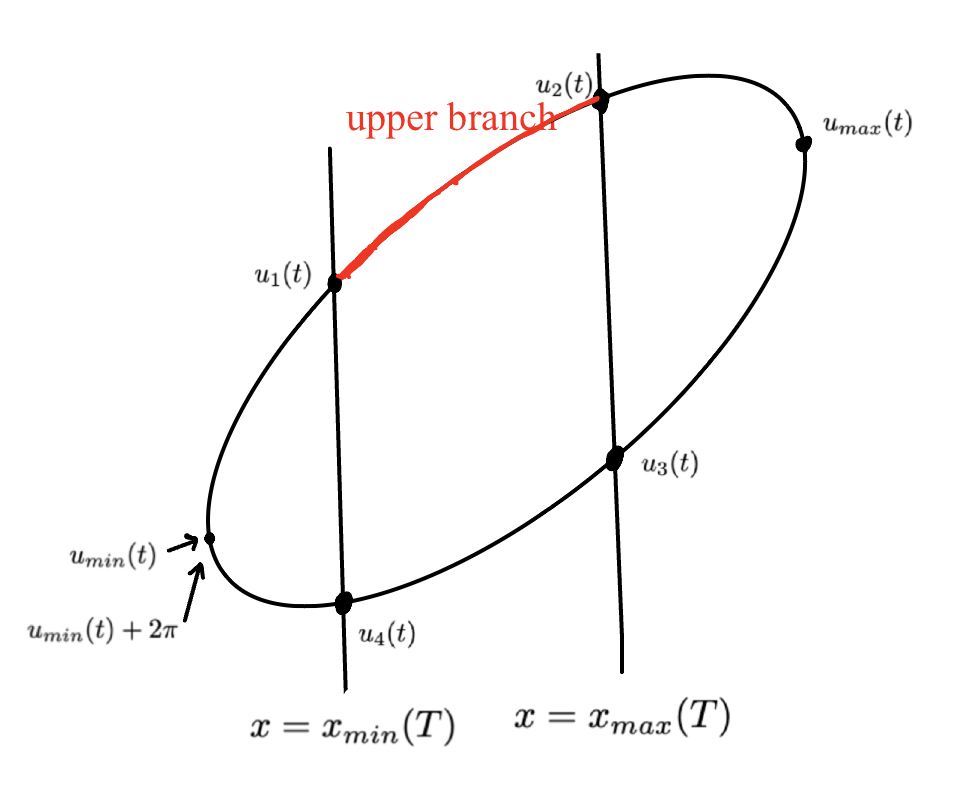} 
\caption{The upper branch}
\label{upper branch} 
\end{figure}
\begin{lem}
\label{one branch is higher}
If $y(u_1(t),t)>y(u_4(t),t)$, then for any $u_{12}(t)\in[u_1(t),u_2(t)]$ and $u_{34}(t)\in[u_3(t),u_4(t)]$ with $x(u_{12}(t),t)=x(u_{34}(t),t)$, one has that $y(u_{12}(t),t)>y(u_{34}(t),t)$.
\end{lem}
\begin{proof}
If this lemma were not true, then $y(u_{12}(t),t)<y(u_{34}(t),t)$ because the projection curve $\gb\att$ is embedded.

By the intermediate value theorem and the condition $y(u_1(t),t)>y(u_4(t),t)$, we are able to find a self-intersection of $\gb\att$, which gives a contradiction.
\end{proof}

In the scenario described by Lemma \ref{one branch is higher}, we would say that the component $\{\gamma(u,t)|u_1(t)\leq u\leq u_2(t)\}$ has larger values of $y$ than the component $\{\gamma(u,t)|u_3(t)\leq u\leq u_4(t)\}$. There always exists a component with larger values of $y$ since the projection curve $\gb\att$ is embedded. 
\begin{defi}
We define the \emph{upper branch} of $\gamma\att$ along the direction $(1,0,0,\cdots,0)$ on the interval $[x_{min}(T), x_{max}(T)]$ to be the component of the set in equation (\ref{the set that is the union of the upper and lower branch}) with larger values of $y$.
\end{defi}
See Figure \ref{upper branch} for an illustration of the upper branch.

\subsection*{The barrier}
We will essentially depend on the following barrier:
\begin{lem}
\label{construction of the barrier}
For $M>0$, consider the upper branch of $\gamma\att$ along direction $(1,0,\cdots,0)$ on the interval $[0,M]\subset[x_{min}(T), x_{max}(T)]$, and let us construct a barrier
\begin{equation}
    \varphi(u, t):=\epsilon e^{-\lambda t} \sin \left(\frac{\pi x(u, t)}{M}\right), \hspace{0.5cm}
    \theta:=\frac{\pi x(u, t)}{M} \in[0, \pi],
\end{equation} 
    where $\lambda\geq\frac{\pi^2}{M^2}$ and $\epsilon>0$.
    One has
     $\varphi_t-\varphi_{ss}\leq 0$ on $x\in \left[0,M\right]$.
\end{lem}
\begin{proof}
    By direct computation:
    $$\varphi_t=\epsilon e^{-\lambda t}\left((\cos\theta) \theta_t-\lambda \sin \theta\right)$$
    $$\varphi_s=\epsilon e^{-\lambda t} (\cos \theta) \theta_s$$
    $$\varphi_{ss}=\epsilon e^{-\lambda t} (-\sin \theta (\theta_s)^2+(\cos \theta) \theta_{ss})$$
Thus,
    $$\varphi_t-\varphi_{ss}=\epsilon e^{-\lambda t} (\sin \theta ((\theta_s)^2-\lambda)+(\cos \theta)(\theta_t- \theta_{ss})).$$
    Notice $$\theta_t- \theta_{ss}=\frac{\pi}{M}(x_t-x_{ss})=0$$
    and since $\lambda\geq\frac{\pi^2}{M^2}$,  $$(\theta_s)^2-\lambda\leq\frac{\pi^2x_s^2}{M^2}-\frac{\pi^2}{M^2}\leq0.$$ 
Therefore,
     $$\varphi_t-\varphi_{ss}\leq0.$$
\end{proof}

\begin{lem}
\label{not double line segment}
If the set $D$ contains two different points $p,q$, assuming
     the straight line passing through points $p,q$ is the x-axis, 
     then $D$ has non-empty interior and $y_{max}(t)> y_{max}(T)>0,\forall t\in\left[0,T\right)$.
\end{lem}
\begin{proof}
Let $x_p\leq x_q$ be the $x$-coordinates of the points $p,q$.

Let us assume $x_p=0$. Then $M:=x_q$ is positive since $p,q$ are two distinct points on the $x$ axis.

By equation (\ref{limit domain D}), the definition of $D$, we know that $p,q\in D(t), \forall t\in[0,T)$. So $x_{min}(t)<x_p<x_q<x_{max}(t)$. By taking $t\rightarrow T$,
$$x_{min}(T)\leq x_p<x_q\leq x_{max}(T).$$
    
Thus we can consider the upper branch of $\gamma\att$ along the direction of $x$ on $[0,M]=[x_p,x_q]\subset[x_{min}(T),x_{max}(T)]$. We may assume we are analyzing the branch 
\begin{equation}
\{\gamma(u,t)|u_1(t)\leq u\leq u_2(t)\},
\end{equation}
where $u_1(t),u_2(t)$ are as labelled in equation (\ref{lablelling of u_i(t)}).

Let us denote by $u_p(t), u_q(t)$ the values of $u$ such that 
\begin{align}
x(u_p(t),t)=x_p\quad y(u_p(t),t)>0\\
x(u_q(t),t)=x_q\quad y(u_q(t),t)>0.
\end{align}
And we may assume 
\begin{equation}
    u_1(t)\leq u_p(t)\leq u_q(t)\leq u_2(t)
\end{equation}

Let $\varphi$ be the barrier in Lemma \ref{construction of the barrier}, because $y_t=y_{ss}$, 
$$(y-\varphi)_t\geq(y-\varphi)_{ss}, \hspace{0.5cm}
u\in[u_p(t), u_q(t)].$$
    
Since points $p,q$ are on the $x$-axis, $y_p=y_q=0$, where $y_p,y_q$ are the $y$-coordinates of the points $p,q$.

We are analyzing the upper branch, thus on the boundary $$y(u_p(t),t)>0 =y_p\text{ and }y(u_q(t),t)>0=y_q$$
    and we can choose $\epsilon$ (coefficient of $\varphi$) small such that
$$y(u,0)\geq\varphi(u,0)\text{ for }u\in[u_p(0), u_q(0)].$$

Thus on the upper branch, by the maximum principle, $\forall t\in[0,T)$ $$y(u,t)\geq\varphi(u,t) \text{ for }  u\in[u_p(t), u_q(t)].$$ 
    
Hence $$y_{max}(t)\geq \varphi_{max}(t)=\epsilon e^{-\lambda t}\geq\epsilon e^{-\lambda T}>0,$$
    
    where we use the fact that $T<+\infty$ due to Corollary \ref{singularity time is finite}.
    
    Therefore $$y_{max}(T)=\lim\limits_{t\rightarrow T}y_{max}(t)\geq\epsilon e^{-\lambda T}>0.$$
\end{proof}

At this point, Proposition \ref{long time behaviour under three point condition} can be proved by combining Lemma \ref{not double line segment} and \cite[Theorem 3.59]{hättenschweiler2015curve}, which replies on a generalization \cite[Theorem 3.49]{hättenschweiler2015curve} of \cite[Theorem 9.1]{angenent1990parabolic}. Next we present our self-contained proof of Proposition \ref{long time behaviour under three point condition} by establishing gradient bounds using the barrier again.

\subsection*{Gradient bounds}
\begin{lem}
\label{gradient bounds for x direction}
For a constant $\delta>0$ such that $x_{max}(T)-x_{min}(T)>2\delta$,
we can bound $|x_s|$ from below  on a smaller interval
$$\inf\limits_{\substack{x\in[x_{min}(T)+\delta,x_{max}(T)-\delta]\\ t\in[0,T)}}|x_s|>0,$$

and hence bound $|\gamma_x|=|\frac{\gamma_s}{x_s}|=|\frac{1}{x_s}|$ from above
$$\sup\limits_{\substack{x\in[x_{min}(T)+\delta,x_{max}(T)-\delta]\\ t\in[0,T)}}|\gamma_x|<+\infty.$$
\end{lem}

\begin{proof}
We may assume $x_{min}(T)=0$. Then $M:=x_{max}(T)>2\delta$.

   Let $\varphi$ be as in Lemma \ref{construction of the barrier}, so $\varphi_t-\varphi_{ss}\leq 0, x\in \left[0,M\right]$.

We can consider the upper branch of $\gamma\att$ on $x\in \left[0,M\right]$ and may assume $x_s\att>0$. The other $x_s<0$ branch can be analyzed similarly.
Say the upper branch is parametrized by $u\in[u_1(t), u_2(t)]$, where $u_1(t),u_2(t)$ are as labelled in equation (\ref{lablelling of u_i(t)}).

    Considering $x_{st}=x_{ts}+k^2x_s=x_{sss}+k^2x_s\geq x_{sss}$, which implies that
    \begin{equation}
    \label{compare with barriers}
        (x_s-\varphi)_t\geq (x_s-\varphi)_{ss}, \text{ for }u\in[u_1(t), u_2(t)].
    \end{equation}
    On the boundary $$x_s(u_1(t),t)> 0=\varphi(x=0,t) \text{ and }x_s(u_2(t),t)> 0=\varphi(x=M,t)$$ 
    and we can choose $\epsilon$ small such that at initial time,
    $$x_s\geq\varphi\text{ on } u\in[u_1(t), u_2(t)].$$

    Thus, by the maximum principle, $\forall t\in\left[0,T\right)$,
    $$x_s\geq\varphi \text{ on } u\in[u_1(t), u_2(t)].$$
    
    Notice for $\delta>0$ and $x\in\left[\delta,M-\delta\right]$, $$\sin(\theta)\geq\sin(\frac{\pi\delta}{M})>0.$$
    
    So $$x_s\geq\varphi\geq\epsilon e^{-\lambda T}\sin(\frac{\pi\delta}{M})>0,$$
    
    where $T<+\infty$ due to Corollary \ref{singularity time is finite} plays a crucial role.
\end{proof}

The previous lemma works for vectors other than $(1,0,\cdots,0)$ with identical proof that we omit. That is to say:
\begin{lem}
\label{gradient bounds}
For any unit horizontal vector $\bm{v}$ and any constant $\delta>0$ such that $$(\bm{v}\cdot\gamma)_{max}(T)-(\bm{v}\cdot\gamma)_{min}(T)>2\delta,$$
we can bound the gradient:
\begin{equation}
\sup\limits_{\substack{  \\ (u,t)\in J_\delta}}\frac{|\gamma_s(u,t)|}{|(\bm{v}\cdot\gamma)_s(u,t)|}
=\sup\limits_{\substack{  \\ (u,t)\in J_\delta}}\frac{1}{|(\bm{v}\cdot\gamma)_s(u,t)|}<+\infty,
\end{equation}
where $J_\delta=\{(u,t)|t\in\left[0,T\right), \bm{v}\cdot\gamma(u,t)\in\left[(\bm{v}\cdot\gamma)_{min}(T)+\delta,(\bm{v}\cdot\gamma)_{max}(T)-\delta\right]\}$.\\ 
\end{lem}

\subsection*{Proof of Proposition \ref{long time behaviour under three point condition}}
\begin{lem}
\label{uniform curvature}
   If $D$ contains two different points, then the curvature of $\gamma(\cdot,t)$ is uniformly bounded, $\forall t\in\left[0,T\right)$.
\end{lem}

\begin{proof}
    Pick $p_0,q_0\in D$ as in Lemma \ref{p0q0}. 
    We may assume points $p_0,q_0$ are on the $x$-axis and $x_{p_0}<x_{q_0}$, where  $x_{p_0},x_{q_0}$ stand for the $x$ coordinates of these two points. By choice of $p_0,q_0$, $x_{min}(T)=x_{p_0}$ and $x_{max}(T)=x_{q_0}$.
    
    Since $D$ contains two different points, we can use Lemma \ref{not double line segment} for $p_0,q_0$. As a result, $y_{max}(T)>0$, $|y_{min}(T)|>0$ and $D$ has non-empty interior. Thus $\gb(\cdot,T)=\partial D$ is a Lipschitz continuous convex curve.
    
    Pick a constant $\delta_1>0$ such that $$2\delta_1<min\{y_{max}(T),|y_{min}(T)|\}.$$
    
    Notice that the intersection of the line $x=x_{min}(T)$ with the curve $\gb(\cdot,T)$ is exactly the single point $p_0$, 
    otherwise we get a contradiction against the choice of points $p_0,q_0$: $d(p_0,q_0)=\sup\limits_{p,q\in D}d(p,q)$.
    
    So we can pick a constant $\delta_2>0$ such that 
    $$|y|({x=x_{min}(T)+2\delta_2},t=T)<min\{y_{max}(T),|y_{min}(T)|\}-2\delta_1$$
    and 
    $$|y|({x=x_{max}(T)-2\delta_2},t=T)<min\{y_{max}(T),|y_{min}(T)|\}-2\delta_1.$$
    
    Combining Lemma \ref{uniform convergence}, there exists a $t_0$,
    such that $\forall t\in\left[t_0,T\right)$, the sets
    $$ \{u|x(u,t)\in\left[x_{min}(T)+2\delta_2,x_{max}(T)-2\delta_2\right]\}$$ 
    and
    $$ \{u|y(u,t)\in\left[y_{min}(T)+2\delta_1,y_{max}(T)-2\delta_1\right]\}$$ covers
    the whole $S^1$.
    
    Next we apply Lemma \ref{gradient bounds} and Proposition \ref{curvature bounds} 
    to direction along $x$ axis choosing $\delta=\delta_2$ and $y$ axis choosing $\delta=\delta_1$. These give the global curvature bounds.
\end{proof}

\begin{rmk}
    In previous proof, picking $p_0,q_0\in D$ as in Lemma \ref{p0q0} is important.
    For example, when $D$ is a square and we choose $x$ axis to be parallel to sides of the square,
    it would not be possible to cover the whole curve by graphs over $x$ axis and $y$ axis, 
    no matter how small we choose $\delta_1,\delta_2$. In this case, 
    picking $p_0,q_0\in D$ as in Lemma \ref{p0q0},
    we are choosing a diagonal of the square to be $x$ axis, which enables the proof to proceed.
\end{rmk}

\begin{proof}[Proof of Proposition \ref{long time behaviour under three point condition}]
    Combine Lemma \ref{bounded curvature can extend flow} and Lemma \ref{uniform curvature},
     $D$ contains 
     at most one point.
     By Lemma \ref{properties of D}, $D$  is nonempty, thus contains
     exactly one point.

    Thus $\gb(\cdot,T)=\partial D=D$, which proves 
    Proposition \ref{long time behaviour under three point condition}.
\end{proof}

\bibliographystyle{alpha}
\bibliography{spaceCSF}
\end{document}